\documentclass[10pt,a4paper]{amsart}
\usepackage[utf8]{inputenc}
\usepackage[T1]{fontenc}
\usepackage[english]{babel}
\usepackage{amsmath,amssymb,amsthm}
\usepackage{bbm}

\usepackage[pdftex]{color}
\usepackage[bookmarks=true,hyperindex,pdftex,colorlinks, citecolor=MidnightBlue,linkcolor=MidnightBlue, urlcolor=MidnightBlue]{hyperref}
\usepackage[dvipsnames]{xcolor}

\usepackage{enumerate}

\usepackage{ulem} 

\makeatletter
\@namedef{subjclassname@2020}{%
  \textup{2020} Mathematics Subject Classification}
\makeatother

\theoremstyle{plain}
\newtheorem{teo}{Theorem}[section]
\newtheorem{prop}[teo]{Proposition}
\newtheorem{proposition}[teo]{Proposition}

\newtheorem{theorem}[teo]{Theorem}
\newtheorem{cor}[teo]{Corollary}
\newtheorem{lemma}[teo]{Lemma}

\theoremstyle{definition}

\newtheorem{prob}[teo]{Problem}
\newtheorem{definition}[teo]{Definition}

\theoremstyle{remark}
\newtheorem{remark}[teo]{Remark}
\newtheorem{remarks}[teo]{Remarks}

 \newcommand{\Id}{\mathrm{Id}}

\newcommand{\K}{\mathbb{K}}
\newcommand{\R}{\mathbb{R}}
\newcommand{\C}{\mathbb{C}}
\newcommand{\T}{\mathbb{T}}

\newcommand{\Li}{\mathcal{L}}

\DeclareMathOperator{\co}{conv}
\DeclareMathOperator{\conv}{conv}
\DeclareMathOperator{\aconv}{aconv}

\DeclareMathOperator{\lin}{span}
\DeclareMathOperator{\re}{Re}
\DeclareMathOperator{\OF}{OF}
\DeclareMathOperator{\sign}{sign}
\DeclareMathOperator{\exposatx}{ext_x^+}

\DeclareMathOperator{\Lip}{Lip_0}
\DeclareMathOperator{\justLip}{Lip}

\newcommand{\cconv}{\overline{\conv}}
\newcommand{\acconv}{\overline{\aconv}}

\renewcommand{\geq}{\geqslant}
\renewcommand{\leq}{\leqslant}

\newcommand{\pten}{\ensuremath{\widehat{\otimes}_\pi}}
\newcommand{\ptensN}{\ensuremath{\widehat{\otimes}_{\pi,s,N}}}
\newcommand{\ext}{\operatorname{ext}}

\newcommand{\supp}{\operatorname{supp}}

\begin{document}

\title[Daugavet property projective symmetric tensor products]{Daugavet property in projective symmetric tensor products of Banach spaces}

\author[M.~Mart\'in]{Miguel Mart\'in}
\address[Mart\'{\i}n]{Departamento de An\'{a}lisis Matem\'{a}tico, Facultad de Ciencias, Universidad de Granada, E-18071-Granada,
Spain\newline
	\href{http://orcid.org/0000-0003-4502-798X}{ORCID: \texttt{0000-0003-4502-798X} }}
\email[Mart\'{\i}n]{mmartins@ugr.es}
\urladdr{\url{https://www.ugr.es/local/mmartins}}

\author[A.~Rueda Zoca]{Abraham Rueda Zoca}
\address[Rueda Zoca]{Universidad de Murcia, Departamento de Matem\'aticas, Campus de Espinardo 30100 Murcia, Spain} \email{abrahamrueda@ugr.es}
\urladdr{\url{https://arzenglish.wordpress.com}}

\thanks{First author partially supported by projects PGC2018-093794-B-I00 (MCIU/AEI/FEDER, UE), A-FQM-484-UGR18 (Universidad de Granada and Junta de Analuc\'{\i}a/FEDER, UE), and FQM-185 (Junta de Andaluc\'{\i}a/FEDER, UE). Second author partially supported by PGC2018-093794-B-I00 (MCIU/AEI/FEDER, UE), A-FQM-484-UGR18 (Universidad de Granada and Junta de Analuc\'{\i}a/FEDER, UE), and FQM-185 (Junta de Andaluc\'{\i}a/FEDER, UE) }

\date{October 29th, 2020}

\keywords{Daugavet property; polynomial Daugavet property; symmetric tensor product; projective tensor product; $L_1$-preduals}
\subjclass[2020]{Primary 46B04; Secondary 46B20; 46B25; 46B28; 46G25}

\begin{abstract}
We show that all the symmetric projective tensor products of a Banach space $X$ have the Daugavet property provided $X$ has the Daugavet property and either $X$ is an $L_1$-predual (i.e.\ $X^*$ is isometric to an $L_1$-space) or $X$ is a vector-valued $L_1$-space. In the process of proving it, we get a number of results of independent interest. For instance, we characterise ``localised'' versions of the Daugavet property (i.e.\ Daugavet points and $\Delta$-points introduced in \cite{ahlp}) for $L_1$-preduals in terms of the extreme points of the topological dual, a result which allows to characterise a polyhedrality property of real $L_1$-preduals in terms of the absence of $\Delta$-points and also to provide new examples of $L_1$-preduals having the convex diametral local diameter two property. These results are also applied to nicely embedded Banach spaces (in the sense of \cite{wernerpredu}) so, in particular, to function algebras. Next, we show that the Daugavet property and the polynomial Daugavet property are equivalent for $L_1$-preduals and for spaces of Lipschitz functions. Finally, an improvement of recent results in \cite{rtv} about the Daugavet property for projective tensor products is also got.
\end{abstract}

\maketitle

\section{Introduction}

A Banach space $X$ is said to have the \emph{Daugavet property} if every rank-one operator $T\colon X\longrightarrow X$ satisfies the so-called \emph{Daugavet equation}:
\begin{equation*}\label{DE}\tag{\textrm{DE}}
\Vert \Id + T\Vert=1+\Vert T\Vert,
\end{equation*}
where $\Id\colon X\longrightarrow X$ denotes the identity operator (and then the equality actually holds for all weakly compact operators). This property comes from the 1963 work of I.~Daugavet \cite{dau} in which the author proved that every compact operator on $C[0,1]$ satisfies the Daugavet equation. Since then, a big effort has been done in order to give more examples of spaces enjoying this property, and also in order to understand its strong connection with different geometrical properties of Banach spaces (see \cite{ikw,kkw,kssw,kw,shv,werner,woj} and references therein). Let us mention that the list of examples of spaces with the Daugavet property includes $C(K)$ spaces when the compact Hausdorff topological space $K$ is perfect, $L_1(\mu)$ and $L_\infty(\mu)$ when the positive measure $\mu$ is atomless (actually, arbitrary vector valued versions of these three kind of spaces work), and the disk algebra, among others. It is of special interest the celebrated characterisation of the Daugavet property given in \cite[Lemma 2.1]{kssw} in terms of a geometric condition of the slices of the unit ball of the Banach space (see the paragraph after Definition~\ref{def:Daugavet} for details). This characterisation has allowed to obtain big progresses on the Daugavet property by making use of techniques coming from the geometry of Banach spaces. A key application of the theory is that a Banach space with the Daugavet property cannot be embedded into a Banach space with unconditional basis \cite{kssw}, extending the classical result by A.~Pe\l czy\'{n}ski  for $L_1[0,1]$ (and so for $C[0,1]$).

One of the oldest questions that nowadays remains open concerning the Daugavet property (explicitly posed in \cite[Section 6, Question (3)]{werner}) is whether $X\pten Y$ has the Daugavet property if $X$ and $Y$ has the Daugavet property. Actually, the original question also asked if $X\pten Y$ has the Daugavet property if $X$ or $Y$ has the Daugavet property, but it was quickly answered in the negative in \cite[Corollary 4.3]{kkw}  (see \cite{llr2} for a counterexample failing even a weaker property than the Daugavet property). Very recently, in \cite[Theorem 1.2]{rtv}, it has been proved that $X\pten Y$ has the Daugavet property if $X$ and $Y$ are $L_1$-preduals with the Daugavet property. The proof relies on an strengthening of the Daugavet property, that the authors of \cite{rtv} named the operator Daugavet property (see Definition~\ref{def:ODP}), which is satisfied by $L_1$-preduals with the Daugavet property thanks to the possibility of extending compact operators on them, a classical result by J.\ Lindenstrauss \cite{linds}. In the final section of \cite{rtv}, the operator Daugavet property is also applied to give non-trivial examples of symmetric projective tensor products with the Daugavet property. More precisely, it is proved in \cite[Proposition 5.3]{rtv} that $\ptensN C(K)$ has the Daugavet property if $K$ is a compact Hausdorff topological space without isolated points and $N$ is an odd positive integer. In view of the non-symmetric case, it is a natural question (suggested in the paragraph after Remark 5.2 in \cite{rtv}) whether $\ptensN X$ has the Daugavet property if $X$ is an $L_1$-predual with the Daugavet property.

The main aim of this paper is to provide a positive answer to the previous question and also to give completely different examples of symmetric tensor products with the Daugavet property. Actually, as a consequence of the results of Section~\ref{section:WODP}, we obtain the following theorem.

\begin{theorem}\label{theo:tensorsymejemplos}
Let $N\in\mathbb N$. Then, the space $\ptensN X$ has the Daugavet property in the following cases:
\begin{enumerate}[(1)]
    \item when $X$ is an $L_1$-predual with the Daugavet property.
    \item when $X=L_1(\mu,Y)$, for an atomless $\sigma$-finite positive measure $\mu$ and a (non-zero) Banach space $Y$.
\end{enumerate}
\end{theorem}

Note that item (1) extends \cite[Proposition 5.3]{rtv} to general $L_1$-preduals, whereas item (2) provides a different kind of examples of symmetric projective tensor products with the Daugavet property.

In the way to prove the above result, we develop a number of techniques and we get a number of results which are of independent interest. Let us explain the contents of the papers.

We devote Section~\ref{sect:Notation} to give the necessary notation and preliminary results needed for the rest of the paper. Next, in Section~\ref{section:L1predDeltapoints} we make a deep study of the Daugavet property for $L_1$-preduals with the Daugavet property which extends the characterisation given in \cite{bm} (based on the results of \cite{wernerpredu}). Actually, the results are ``localised'' in the sense introduced very recently in \cite{ahlp} of the study of the Daugavet-points and $\Delta$-points in Banach spaces (see Definition~\ref{def:Daugavetanddeltapoints}). We characterise in Theorem~\ref{theo:carapuntodaugapreduL1} these kind of points for an $L_1$-predual in terms of the behaviour of the extreme points of the dual ball, and also in terms of the possibility of getting special $c_0$-sequences in the bidual space. This characterisation generalises previously known results from \cite{ahlp}. The main tool to prove the theorem is the use of $L$-projections techniques, so it is actually true for nicely embedded Banach spaces (Proposition~\ref{prop:nicelyembedded}), in particular, to function algebras. The section ends with a discussion on the relationship between our results and polyhedrality of real $L_1$-preduals and with applications of the results to the convex diametral local diameter two property for $L_1$-preduals (Corollaries \ref{cor:suff-convexDLD2P} and \ref{cor:suff-convexDLD2P-particularcases}) and for nicely embedded Banach spaces (Corollary~\ref{coro:nicellyembedded-convexDLD2P}) so, in particular, for function algebras. These results extend again results from \cite{ahlp} and provide with new examples of Banach spaces with the convex diametral local diameter two property.

Section~\ref{sect:Polynomial} deals with the polynomial Daugavet property, a property stronger than the Daugavet property which requires Eq.~\eqref{DE} to hold for weakly compact polynomials instead of just linear operators (see Definition~\ref{def:polynomialDaugavet}). Using the results of Section~\ref{section:L1predDeltapoints}, we show that $L_1$-preduals with the Daugavet property actually fulfill the polynomial Daugavet property (Theorem~\ref{theorem:polyDPpreduL1}), extending the result from \cite{cgmm} that this happens for $C(K)$ spaces. This result will be a key tool in the way of proving in Section~\ref{section:WODP} item (1) of Theorem~\ref{theo:tensorsymejemplos}. Besides, we include the analogous result to Theorem~\ref{theorem:polyDPpreduL1} for spaces of Lipschitz functions, see Proposition~\ref{prop:Lipschitz}.

Finally, we devote Section~\ref{section:WODP} to the last steps to prove Theorem~\ref{theo:tensorsymejemplos}. We introduce in Definition~\ref{def:WODP} a property called weak operator Daugavet property (WODP), which is (formaly) weaker that the ODP but still implies the Daugavet property.  We show that the WODP is stable by projective tensor products (Theorem~\ref{theorem:WODP}), a promising result in connection with a possible positive answer to \cite[Section 6, Question (3)]{werner}. Observe that this result improves those of \cite{rtv}. Furthermore, we introduce a mix of the WODP and the polynomial Daugavet property, which we call the polynomial weak operator Daugavet property (polynomial WODP in short), see Definition~\ref{defi:polywodp}, which implies both of them. We prove in Propositions \ref{prop:WODPpolinomial-predual} and \ref{prop:WODPpolinomial-L1vectorvalued} that both $L_1$-preduals with the Daugavet property and $L_1(\mu,Y)$, for a non-atomic measure $\mu$ and any non-zero Banach space $Y$, enjoy the polynomial WODP. Finally, in Theorem~\ref{theo:polynomialWODPimpliesptensNWODP} we prove that if $X$ has the polynomial WODP, then for every positive integer $N$ $\ptensN X$ has the WODP (so, in particular, the Daugavet property). Putting all together, we obtain the promised proof of Theorem~\ref{theo:tensorsymejemplos}.

\section{Notation and preliminary results}\label{sect:Notation}
We denote by $\mathbb K$ the scalar field, which will always be either $\mathbb R$ or $\mathbb C$, and the set of modulus one scalars by $\T$. Given a Banach space $X$, we denote the closed unit ball and the unit sphere of $X$ by $B_X$ and $S_X$, respectively. The topological dual of $X$ is denoted by $X^*$. Given a closed convex and bounded subset $C$ of $X$, a \emph{slice of $C$} is the non-empty intersection of $C$ with an open half space. We use the notation
$$
S(C,x^*,\alpha):=\{x\in C\colon \re  x^*(x)>\sup \re  x^*(C)-\alpha\}
$$
where $x^*\in X^*$ and $\alpha>0$. Note that every slice of $C$ can be written in the above form. We write $\ext(C)$ to denote the set of extreme points of $C$. Given a subset $B\subset X$, the convex hull and the absolutely convex hull of $B$ are denoted, respectively, by $\conv(B)$ and $\aconv(B)$. The closure of these two sets is denoted by $\cconv(B)$ and $\acconv(B)$, respectively.

Let us recall the definition of the Daugavet property from \cite{kssw}

\begin{definition}[\textrm{\cite{kssw}}]\label{def:Daugavet}
A Banach space $X$ is said to have the \emph{Daugavet property} if every rank-one operator $T\colon X\longrightarrow X$ satisfies the equation
$$
\Vert \Id+T\Vert=1+\Vert T\Vert,
$$
where $\Id\colon X\longrightarrow X$ denotes the identity operator.
\end{definition}

As commented in the introduction, examples of Banach spaces with the Daugavet property are $C(K)$ spaces when the compact Hausdorff space $K$ has no isolated points, $L_1(\mu)$ when the positive measure $\mu$ has no atoms, the disk algebra, and non-atomic $C^*$-algebras, among many others. We refer the reader to the papers \cite{BM_JFA2005,ikw,kkw,kssw,kw,shv,werner,woj} and references therein for background. The following geometric characterisation of the Daugavet property, given in \cite[Lemma 2.1]{kssw}, is well known and will be freely used throughout the text without any explicit mention.
\begin{quote}
A Banach space $X$ has the Daugavet property if, and only if, for every $\varepsilon>0$, every point $x\in S_X$ and every slice $S$ of $B_X$, there exists a point $y\in S$ such that $\Vert x+y\Vert>2-\varepsilon$.
\end{quote}

Given two Banach spaces $X$ and $Y$, we denote by $\Li(X,Y)$ the space of bounded linear operators $T\colon X\longrightarrow Y$. We denote by $\mathcal B(X,Y)$ the space of bounded bilinear maps $G\colon X\times Y\longrightarrow \mathbb K$. For $N\in \mathbb{N}$, $\mathcal{P}(^N X,Y)$ is the Banach space of $N$-homogeneous scalar-valued continuous polynomials from $X$ into $Y$ and we write $\mathcal{P}(^0 X,Y)$ for the space of constant functions. The space of all $Y$-valued continuous polynomials is then
$$
\mathcal{P}(X,Y):=\left\{\sum_{k=0}^n P_k\colon n\in \mathbb{N},\ P_k\in \mathcal{P}(^k X,Y)\ \forall k=1,\ldots,n\right\}.
$$
Recall that $\mathcal{P}(X,Y)$ is a normed space when endowed with the norm $\|P\|=\sup_{x\in B_X} \|P(x)\|$ for every $P\in \mathcal{P}(X,Y)$.
We simply write $\mathcal{P}(^N X)$ and $\mathcal{P}(X)$ for, respectively, $\mathcal{P}(^N X,\K)$ and $\mathcal{P}(X,\K)$.

Recall that the \emph{projective tensor product} of $X$ and $Y$, denoted by $X\pten Y$, is the completion of the algebraic tensor product $X\otimes Y$ under the norm given by
\begin{equation*}
   \Vert u \Vert :=
   \inf\left\{
      \sum_{i=1}^n  \Vert x_i\Vert\Vert y_i\Vert
      \colon u=\sum_{i=1}^n x_i\otimes y_i
      \right\}.
\end{equation*}
It follows easily from the definition  that $$B_{X\pten Y}=\cconv(B_X\otimes B_Y)
=\cconv(S_X\otimes S_Y).$$
It is well known that
$(X\pten Y)^*=\Li(X,Y^*)=\mathcal{B}(X,Y)$ see \cite[p.~27]{DeFl} for instance. We refer the reader to \cite{DeFl,rya} for a detailed treatment of tensor product spaces.

Given a Banach space $X$, the \emph{($N$-fold) projective symmetric tensor product} of $X$, denoted by
$\widehat{\otimes}_{\pi,s,N} X$, is defined as the completion of the space
$\otimes^{s,N}X$ under the norm
\begin{equation*}
   \Vert u\Vert:=\inf
   \left\{
      \sum_{i=1}^n \vert \lambda_i\vert \Vert x_i\Vert^N \colon
      u:=\sum_{i=1}^n \lambda_i x_i^N,\, n\in\mathbb N,\, x_i\in X
   \right\}.
\end{equation*}
Notice that $B_{\widehat{\otimes}_{\pi,s,N} X} =\acconv\bigl(\bigl\{x^N\colon x\in S_X\bigr\}\bigr)$ and that $\bigl[\widehat{\otimes}_{\pi,s,N} X\bigr]^*=\mathcal P(^N X)$ (see \cite{flo} for background).

A projection $P\colon X\longrightarrow X$ on a Banach space $X$ is said to be an \emph{$L$-projection} if $\|x\|=\|Px\|+\|x-Px\|$ for every $x\in X$. The range of an $L$-projection is called an \emph{$L$-summand}. The following easy result on $L$-projection is surely well known. We include it here as we have not found any concrete reference, although it follows routinely from \cite[Theorem~I.1.10]{HWW}.

\begin{lemma}\label{lemma:L-projections}
Let $Z$ be a Banach space and let $z_1,\ldots,z_n\in S_Z$ pairwise linearly independent elements such that $\K z_k$ is an $L$-summand of $Z$ for every $k=1,\ldots,n$. For each $k\in \{1,\ldots,n\}$, write $P_k$ for the $L$-projection with range $\K z_k$, so $Z=\K z_k\oplus_1 \ker P_k$. Then, $P_kP_j=0$ when $k\neq j$, $P:=P_1+\cdots + P_n$ is an $L$-projection with kernel $\bigcap\nolimits_{k=1}^n \ker P_k$, and $P(Z)\equiv \ell_1^n$ with $B_{P(Z)}=\aconv\bigl(\{p_1,\ldots,p_n\}\bigr)$. In particular, the points $z_1,\ldots,z_n$ are linearly independent.
\end{lemma}

\begin{proof}
First, fix $k,j$ with $k\neq j$ and use that $P_kP_j=P_jP_k$ by \cite[Theorem~I.1.10]{HWW} to get that $P_kP_j(Z)\subset P_k(Z)\cap P_j(Z)=(\K p_k)\cap (\K p_j)$. As $p_k$ and $p_j$ are linearly independent, we get that $P_kP_j(Z)=0$, that is, $P_kP_j=0$. Now, it follows also from \cite[Theorem~I.1.10]{HWW} that $P=P_1+\cdots+P_n$ is an $L$-projection. It is straightforward to show that $\ker P =\bigcap\nolimits_{k=1}^n \ker P_k$ using that the projections are orthogonal. Finally, it is also immediate that $P(Z)\equiv \ell_1^n$ and that $B_{P(Z)}=\aconv\bigl(\{p_1,\ldots,p_n\}\bigr)$.
\end{proof}

By an \emph{$L_1$-predual} we mean a Banach space $X$ such that $X^*\equiv L_1(\mu)$ for certain measure $\mu$. We refer the reader to the book \cite{Lacey} and the seminal paper \cite{linds} for background on these spaces. Also, we refer to \cite{bm,wernerpredu} for background on $L_1$-preduals with the Daugavet property. Recall that the extreme points of the unit ball of an $L_1(\mu)$ space are of the form $\theta \frac{\chi_A}{\mu(A)}$ where $\theta\in \T$ and $A$ is an atom of $\mu$ with $0<\mu(A)<\infty$. It is readily follows that when $f_0\in \ext(B_{L_1(\mu)})$, then $\K f_0$ is an $L$-summand of $L_1(\mu)$. Actually, $L_1(\mu)=\K f_0 \oplus_1 Z$ where $Z$ is just the subspace of those functions of $L_1(\mu)$ whose support do not intersect $A$ and the projection is given by $P(f)=\frac{1}{\mu(A)}\int f\chi_A$ for every $f\in L_1(\mu)$. With this in mind, the particular case of Lemma~\ref{lemma:L-projections} in which $Z=L_1(\mu)$ and $z_1,\ldots,z_n$ are pairwise linearly independent extreme points of $B_Z$ is immediate to prove directly.

\section{Daugavet-points and $\Delta$-points in $L_1$-preduals} \label{section:L1predDeltapoints}

Our main goal in this section is to study $L_1$-preduals with the Daugavet property, showing some characterisations which will be the key in
Section~\ref{sect:Polynomial} to get that they have the polynomial Daugavet property and in Section~\ref{section:WODP} to get the stability of the Daugavet property by symmetric tensor product of them. We need some notation which allows to ``localise'' the Daugavet property in the sense that has been recently done in \cite{ahlp}.

\begin{definition}[\textrm{\cite{ahlp}}]\label{def:Daugavetanddeltapoints}
Given a Banach space $X$, a point $x\in S_X$ is said to be:
\begin{enumerate}[(a)]
    \item a \emph{Daugavet-point} if, for every slice $S$ of $B_X$ and every $\varepsilon>0$ there exists $y\in S$ with $\Vert x-y\Vert>2-\varepsilon$, equivalently, if $$B_X=\cconv\bigl(\{y\in B_X\colon \|x-y\|>2-\varepsilon\}\bigr) \ \ \text{for every $\varepsilon>0$}.$$
    \item a \emph{$\Delta$-point} if, for every $\varepsilon>0$ and every slice $S$ of $B_X$ containing $x$, there exists $y\in S$ with $\Vert x-y\Vert>2-\varepsilon$, equivalently, if $$x\in \cconv\bigl(\{y\in B_X\colon \|x-y\|>2-\varepsilon\}\bigr) \ \ \text{for every $\varepsilon>0$}.$$
\end{enumerate}
\end{definition}

It is clear that a Banach space $X$ has the Daugavet property if, and only if, every element of $S_X$ is a Daugavet-point (see \cite[Corollary 2.3]{werner}). The case that every element of $S_X$ is a $\Delta$-point is known to be equivalent to a property called the diametral local diameter two property, see \cite[Proposition 1.1]{ahlp}. It is immediate that every Daugavet-point is a $\Delta$-point but, in general, a $\Delta$-point does not need to be a Daugavet-point \cite[Example 4.7]{ahlp}. See \cite{ahlp,HallerPirkWeeorg} for background and motivation for the study of Daugavet-points and $\Delta$-points.

Let us start with the following characterisation of the Daugavet-points and $\Delta$-points in $L_1$-preduals. Given a Banach space $X$ and $x\in S_X$, we write
$$
D(x):=\{x^*\in S_{X^*}\colon x^*(x)=1\}=\{x^*\in S_{X^*}\colon \re x^*(x)=1\}
$$
and we write
$$
\exposatx(B_{X^*}):=\{x^*\in \ext(B_{X^*})\colon \re x^*(x)=|x^*(x)|\}.
$$
Observe that $\T \exposatx(B_{X^*})=\ext(B_{X^*})$.

From now on, we consider the set $\exposatx(B_{X^*})$ endowed with the restriction of the weak-start topology. Finally, note that two different elements in $\exposatx(B_{X^*})$ for which the value at $x$ is not zero have to be linearly independent.

\begin{theorem}\label{theo:carapuntodaugapreduL1}
Let $X$ be an $L_1$-predual and $x\in S_X$. Then, the following assertions are equivalent:
\begin{enumerate}
\item[(1)] $x$ is a Daugavet-point.
\item[(2)] $x$ is a $\Delta$-point.
\item[(3)] For every $\varepsilon>0$, the set $$\{e^*\in \exposatx(B_{X^*})\colon \re e^*(x)>1-\varepsilon\}$$
is infinite.
\item[(4)] For every $\varepsilon>0$, the set $$\{e^*\in \ext(B_{X^*})\colon |e^*(x)|>1-\varepsilon\}$$
contains infinitely many pairwise linearly independent elements.
\item[(5)] $D(x)\cap [\exposatx(B_{X^*})]'\neq \emptyset$.
\item[(6)] For every $y\in B_X$ there exists a sequence $\{x_n^{**}\}\subseteq B_{X^{**}}$ satisfying that $\limsup\Vert x-x_n^{**} \Vert=2$ and that
    $$
    \left\|\sum_{k=1}^m \lambda_k (x_k^{**}-y)\right\|\leq 2 \max\bigl\{|\lambda_1|,\ldots,|\lambda_m|\bigr\}
    $$
    for every $m\in \mathbb{N}$ and every $\lambda_1,\ldots,\lambda_m\in \K$ (that is, the linear operator $T$ from $c_0$ to $X^{**}$ defined by $T(e_n) = x_n^{**}-y$ for all $n\in \mathbb{N}$ is continuous).
\item[(7)] For every $y\in B_X$ there exists a sequence $\{x_n^{**}\}\subseteq B_{X^{**}}$ satisfying that $\limsup\Vert x-x_n^{**} \Vert=2$ and that $\{x_n^{**}\}\longrightarrow y$ in the weak-star topology.
\end{enumerate}
\end{theorem}

Note that in the real case, assertion (5) of the theorem above is equivalent to $D(x)\cap \bigl[\ext(B_{X^*})\bigr]'\neq \emptyset$.

\begin{proof}
(1)$\Rightarrow$(2) is obvious.

(2)$\Rightarrow$(3). Assume that (3) does not hold and so that there exists $\varepsilon_0>0$ such that the set $$\{e^*\in \exposatx(B_{X^*})\colon e^*(x)>1-\varepsilon_0\}$$ is finite. Then, there exists extreme points $e^*_1,\ldots, e^*_k$ and $\alpha>0$ such that $e^*_i(x)=1$ for $i=1,\ldots,k$ and $\vert e^*(x)\vert\leq 1-\alpha$ if $e^*\in \ext(B_{X^*})\setminus\T \{e^*_1,\ldots, e^*_k\}$.

Define $g:=\frac{1}{k}\sum_{i=1}^k e^*_i$, which is a norm-one functional as $g(x)=1$. Define $S=S(B_X,g,\frac{\alpha}{2k})$. Pick $y\in S$ and let us estimate $\Vert x-y\Vert$. As $\re g(y)>1-\dfrac{\alpha}{2k}$, a convexity argument gives $$\re e^*_i(y)>1-\frac{\alpha}{2}\ \ \text{ for every $i\in\{1,\ldots, k\}$}.
$$
In particular, $\vert e^*_i(x-y)\vert<\frac{\alpha}{2}$ for every $1\leq i\leq k$. Now, since $\T \exposatx(B_{X^*})=\ext(B_{X^*})$, we have that
\[
\begin{split}
\Vert x-y\Vert&=\sup\left\{ \vert e^*(x-y)\vert\colon e^*\in \exposatx(B_{X^*})\right\}\\ & =\max\left\{\max\limits_{1\leq i\leq k} \vert e^*_i(x-y)\vert,\sup\limits_{e^*\notin \{e^*_1,\ldots, e^*_k\}}\vert e^*(x-y)\vert\right\}\\
& \leq \max \left\{\frac{\alpha}{2}, \sup\limits_{e^*\notin \{e^*_1,\ldots, e^*_k\}}\vert e^*(x)\vert +\vert e^*(y)\vert \right\}\\
& \leq \max\left\{\frac{\alpha}{2}, 1+1-\alpha\right\}\leq 2-\alpha.
\end{split}
\]
Since, clearly, $x\in S$ and $y\in S$ was arbitrary, we get that $x$ is not a $\Delta$-point.

(3)$\Leftrightarrow$(4)$\Leftrightarrow$(5) are immediate.

(3)$\Rightarrow$(6). Pick $y\in B_X$. By the assumption, take an infinite set $\{e^*_n\}\subseteq \exposatx(B_{X^*})$ such that $\re e^*_n(x)>1-\frac{1}{n}$ for all $n\in \mathbb{N}$. Observe that the elements of $\{e_n^*\colon n\in \mathbb{N}\}$ are pairwise linearly independent.  Notice that, since $X^*\equiv L_1(\mu)$ for some positive measure $\mu$, being each $e^*_n$ an extreme point of $B_{X^*}$, we may find an $L$-projection $P_n\colon X^*\longrightarrow X^*$ such that $P_n(X^*)=\K e_n^*$. Now, for every $n\in\mathbb N$ define the linear functional $x_n^{**}\colon X^*=\K e^*_n\oplus_1 \ker P_n\longrightarrow \mathbb K$ by
\begin{equation*}\label{eq:proof-Theorem-L1predual-def-xnstarstar} \tag{$\star$}
x_n^{**}(\lambda e^*_n+z^*)=-\lambda+z^*(y).
\end{equation*}
Notice that, since $\|y\|\leq 1$, we have that
$$
\bigl|x_n^{**}(\lambda e^*_n + z^*)\bigr| \leq |\lambda|+\|z^*\|= \|\lambda e^*_n + z^*\|,
$$
so $x_n^{**}$ is continuous and, moreover, $x_n^{**}\in B_{X^{**}}$. Let us now prove that the sequence $\{x_n^{**}-y\}$ satisfies our requirements. Indeed, pick $m\in\mathbb N$ and $\lambda_1,\ldots, \lambda_m\in \mathbb K$. We consider $P=P_1+\cdots+P_m$ and use Lemma~\ref{lemma:L-projections} to get that $P$ is an $L$-projection, that $X^*=P(X^*)\oplus_1 \ker P$, that $\ker P=\bigcap\nolimits_{k=1}^m \ker P_k$, and that $B_{P(Z)}=\aconv\bigl(\{e^*_1,\ldots,e^*_n\}\bigr)$. With this in mind, taking into account that $x_k^{**}(x^*)-x^*(y)=0$ for $k=1,\ldots,m$ whenever $x^*\in \ker P$, we have that
$$
\left\Vert \sum_{k=1}^{m} \lambda_k (x_k^{**}-y)\right\Vert= \sup\limits_{j=1,\ldots,m}\left\vert \sum_{k=1}^m \lambda_k \bigl(x_k^{**}(e^*_j)-e^*_j(y)\bigr)\right\vert.
$$
But now, as $x_k^{**}(e^*_j)-e_j^*(y)=0$ whenever $k,j\in\{1,\ldots,m\}$ with $k\neq j$, it follows that
\begin{align*}
\left\Vert \sum_{k=1}^{m} \lambda_k (x_k^{**}-y)\right\Vert & = \max\limits_{j=1,\ldots,m} \bigl\vert \lambda_j\bigl(x_j^{**}(e^*_j)- e^*_j(y)\bigr)\bigr\vert \\ & \leq \max\limits_{j=1,\ldots,m} \bigl\vert \lambda_j\bigl(|x_j^{**}(e^*_j)| + |e^*_j(y)|\bigr)\bigr\vert\leq 2\max\limits_{1\leq j\leq n}\vert \lambda_j\vert.
\end{align*}
On the other hand, since $$\Vert x-x_n^{**}\Vert\geq  \bigl|e^*_n(x)-x_n^{**}(e^*_n)\bigr|>2-\frac{1}{n},$$ it follows that $\limsup \Vert x-x_n^{**}\Vert=2$, as desired.

(6)$\Rightarrow$(7). It is immediate since the basis $\{e_n\}$ of $c_0$ converges weakly to $0$ and then, so does $\{T(e_n)\}=\{x_n^{**}-y\}$. A fortiori, $\{x_n^{**}\}$ converges to $y$ in the weak-star topology.

(7)$\Rightarrow$(1). Pick $\varepsilon>0$ and a slice $S=S(B_X,g,\alpha)$ of $B_X$, where $g\in S_{X^*}$ and $\alpha>0$. Pick $y\in S$ and consider, by the assumption, a sequence $\{x_n^{**}\}$ in $S_{X^{**}}$ satisfying that $\limsup\Vert x-x_n^{**}\Vert=2$ and that $x_n^{**}\longrightarrow y$ in the weak-star topology. Since $\re g(y)>1-\alpha$, we may find $n\in \mathbb{N}$ large enough so that
$$
\re x_n^{**}(g)>1-\alpha \quad \text{ and }\quad \Vert x-x_n^{**}\Vert>2-\varepsilon.
$$
Now, by the weak-star denseness of $B_X$ in $B_{X^{**}}$ and the lower weak-star semicontinuity of the norm of $X^{**}$, we may find $z\in B_X$ such that
$$
\re g(z)>1-\alpha \quad \text{ and }\quad \Vert x-z\Vert> 2-\varepsilon.
$$
This proves that $x$ is a Daugavet-point, as desired.
\end{proof}

There are several remarks and consequences of the result above which we would like to state. Let us start by presenting some results which are extended by it.

\begin{remarks}\label{remarks:DaugavetDeltapoints}
\begin{enumerate}[(a)]
  \item Theorem~\ref{theo:carapuntodaugapreduL1} extends \cite[Theorem 3.4]{ahlp}, where the equivalences (1)$\Leftrightarrow$(2)$\Leftrightarrow$(5) was given for $C(K)$ spaces. Observe that assertion (5) for a $C(K)$ space can be  written in terms of accumulation points of $K$ using the well-known equivalence between $K$ and the set of extreme points of the unit ball $C(K)^*$ with the weak-star topology. This is how (4) is given in \cite{ahlp}.
  \item Besides, the fact that Daugavet-points and $\Delta$-points are equivalent for general $L_1$-preduals was already known. It is shown in \cite[Theorem~3.7]{ahlp} with an indirect argument which does not need to characterise them. Indeed, they first proved the result for $C(K)$ spaces (item (a) above) and then translated it to arbitrary $L_1$-preduals by an argument depending on the principle of local reflexivity.
\end{enumerate}
\end{remarks}

Our next comment is that Theorem~\ref{theo:carapuntodaugapreduL1} gives an alternative proof of a characterisation of the Daugavet property for $L_1$-preduals given in \cite{wernerpredu} and \cite{bm}. We need some notation. Given a Banach space $X$, we consider the equivalent relation $f\sim g$ if and only if $f$ and $g$ are linearly independent elements of $\ext(B_{X^*})$ and we endow the quotient space $\ext(B_{X^*})/\sim$ with the quotient topology of the weak-star topology. Observe that $f,g\in \ext(B_{X^*})$ are linearly dependent if and only if $f=\theta g$ for some $\theta\in \T$, so the equivalent class of $f\in \ext(B_{X^*})$ identifies with $\T f$.

\begin{cor}[\textrm{\cite[Theorem 3.5]{wernerpredu} and \cite[Theorem 2.4]{bm}}]
\label{corollary:Daugavet-L1preduals}
Let $X$ be an $L_1$-predual. Then, $X$ has the Daugavet property if and only if $\ext(B_{X^*})/\sim$ does not contain any isolated point.
\end{cor}

\begin{proof}
Suppose $X$ does not have the Daugavet property. Then, there is $x\in S_X$ which is not a Daugavet-point, so by Theorem~\ref{theo:carapuntodaugapreduL1} there is $\varepsilon_0>0$ such that, writing $W=\{x^*\in B_{X^*}\colon \re x^*(x)>1-\varepsilon_0\}$, we have that $W\cap \exposatx(B_{X^*})$ is finite. Observe that this implies that only finitely many linearly independent extreme points of $B_{X^*}$ belong to $W$, which is an weak-star open set. This shows that $\ext(B_{X^*})/\sim$ contains isolated points.

Conversely, suppose that the equivalent class of $e^*\in \ext(B_{X^*})$ is isolated in $\ext(B_{X^*})/\sim$. That is, there is a weak-star open set $W$ of $X^*$ containing $e^*$ such that
$\ext(B_{X^*})\cap W\subseteq \T e^*$. By Choquet's lemma, we may suppose that $W$ is a weak-star slice, that is, there is $x\in S_X$ and $\varepsilon>0$ such that
$$
\{x^*\in \ext(B_{X^*})\colon \re x^*(x)> 1-\varepsilon\}\subseteq \T e^*.
$$
But this clearly implies that the set
$$
\{x^*\in \exposatx(B_{X^*})\colon \re x^*(x)>1-\varepsilon\}
$$
contains only one element. So, Theorem~\ref{theo:carapuntodaugapreduL1} gives that $x$ is not a Daugavet-point and so $X$ fails the Daugavet property.
\end{proof}

Next, a sight to the proof of Theorem~\ref{theo:carapuntodaugapreduL1} shows that the hypothesis of $X$ being an $L_1$-predual is only used for the implication (3)$\Rightarrow$(6), being the rest of implications true for general Banach spaces.

\begin{remark}\label{remark:algunasvalenengeneral}
Let $X$ be a Banach space and consider the assertions (1) to (7) of Theorem~\ref{theo:carapuntodaugapreduL1}. Then, the following implications hold:
$$
\text{(1)}\Rightarrow \text{(2)} \Rightarrow \text{(3)} \Leftrightarrow \text{(4)} \Leftrightarrow \text{(5)} \qquad \text{ and }\qquad \text{(6)} \Rightarrow \text{(7)} \Rightarrow \text{(1)}.
$$
\end{remark}

Let us observe that (3)$\Rightarrow$(2) does not hold in general (and so, neither does (3)$\Rightarrow$(6)), as $X=\ell_2$ shows.

Our next remark on Theorem~\ref{theo:carapuntodaugapreduL1} is that it is possible to give a version of it for nicely embedded spaces. Let us introduce some notation. Let $S$ be a Hausdorff topological space, and let $C^b(S)$ be the sup-normed Banach space of all bounded continuous scalar-valued functions. For $s\in \Omega$, the funcional $f\longmapsto f(s)$ is denoted by $\delta_s$.

\begin{definition}[\textrm{\cite{wernerpredu}}]\label{def:nicelyembedded}
A Banach space $X$ is \emph{nicely embedded} into $C_b(S)$ if there is an isometry $J\colon X\longrightarrow C^b(S)$ such that for all $s\in S$ the following properties are satisfied:
\begin{enumerate}[(N1)]
\item For $p_s:=J^*(\delta_s)\in X^*$ we have $\|p_s\|=1$.
\item $\K p_s$ is an $L$-summand in $X^*$.
\end{enumerate}
We will further suppose, for the sake of simplicity and since it can be done in the most interesting examples, that the elements of the set $\{p_s\colon s\in S\}\subset X^*$ are pairwise linearly independent (so, by (N1), they are linearly independent, see Lemma~\ref{lemma:L-projections}).
\end{definition}

We have the following version of Theorem~\ref{theo:carapuntodaugapreduL1}.

\begin{prop}\label{prop:nicelyembedded}
Let $X$ be a Banach space nicely embedded into $C^b(S)$ for which $\{p_s\colon s\in S\}$ is (pairwise) linearly independent and let $x\in S_X$. Then, the following assertions are equivalent:
\begin{enumerate}
\item[(1)] $x$ is a Daugavet-point.
\item[(2)] $x$ is a $\Delta$-point.
\item[(4)] For every $\varepsilon>0$, the set $$\{s\in S\colon |p_s(x)|>1-\varepsilon\}$$
is infinite.
\item[(5)] $D(x)\cap \bigl[\{p_s\colon s\in S\}\bigr]'\neq \emptyset$.
\item[(6)] For every $y\in B_X$ there exists a sequence $\{x_n^{**}\}\subseteq B_{X^{**}}$ satisfying that $\limsup\Vert x-x_n^{**} \Vert=2$ and that
    $$
    \left\|\sum_{k=1}^m \lambda_k (x_k^{**}-y)\right\|\leq 2 \max\bigl\{|\lambda_1|,\ldots,|\lambda_m|\bigr\}
    $$
    for every $m\in \mathbb{N}$ and every $\lambda_1,\ldots,\lambda_m\in \K$ (that is, the linear operator $T$ from $c_0$ to $X^{**}$ defined by $T(e_n) = x_n^{**}-y$ for all $n\in \mathbb{N}$ is continuous).
\item[(7)] For every $y\in B_X$ there exists a sequence $\{x_n^{**}\}\subseteq B_{X^{**}}$ satisfying that $\limsup\Vert x-x_n^{**} \Vert=2$ and that $\{x_n^{**}\}\longrightarrow y$ in the weak-star topology.
\end{enumerate}
\end{prop}

The proof is just an adaptation of the one of Theorem~\ref{theo:carapuntodaugapreduL1}. Actually, as it is noted in Remark~\ref{remark:algunasvalenengeneral}, only (4)$\Rightarrow$(6) has to be proved. To get this implication, we follow the proof of (3)$\Rightarrow$(6) of Theorem~\ref{theo:carapuntodaugapreduL1}, find a sequence $\{s_n\}$ of different points of $S$ such that $|p_{s_n}(x)|>1-\varepsilon$ and, instead of using Eq.~\eqref{eq:proof-Theorem-L1predual-def-xnstarstar}, we define the linear functional $x_n^{**}\colon X^*=\K p_{s_n}\oplus_1 \ker P_n\longrightarrow \mathbb K$ by
\begin{equation*}
x_n^{**}(\lambda p_{s_n}+z^*)=-\lambda\theta_n +z^*(y)
\end{equation*}
where $\theta_n\in \T$ satisfies that $p_{s_n}(x)=\theta_n|p_{s_n}(x)|$ for every $n\in \mathbb{N}$.

The above result applies, for instance, to a \emph{function algebra} $A$ on a compact Hausdorff space $K$, that is, A is a closed subalgebra of a $C(K)$ spaces separating the points of $K$ and containing the constant functions. Indeed, to $A$ it is associated a distinguished subset $\delta A\subset K$, called the Choquet boundary of $A$, defined by
$$
\partial A=\bigl\{k\in K\colon \delta_k|_A \text{ is an extreme point of $B_{A^*}$}\bigr\}.
$$
Then, it is known that $A$ is nicely embedded into $C^b(\partial A)$ (see the proof of \cite[Theorem 3.3]{wernerpredu} for instance). A distinguished example is the disk algebra $\mathbb{A}$, the space of those functions on $C\bigl(\overline{\mathbb{D}}\bigr)$ which are holomorphic on $\mathbb{D}$, endowed with the supremum norm. The Choquet boundary of $\mathbb{A}$ is $\T$, so Proposition~\ref{prop:nicelyembedded} gives an alternative proof of the fact that $\mathbb{A}$ has the Daugavet property from \cite{woj} or \cite{wernerpredu}.

Next, let us relate $\Delta$-points and polyhedrality for real Banach spaces. Recall that a real Banach space is said to be \emph{polyhedral} if the unit balls of all its finite-dimensional subspaces are polytopes (i.e.\ they have finitely many extreme points). There are several versions of polyhedality which have been studied in the literature (see \cite{DurierPapini,FonfVesely} and references therein) of which we would like to emphasise the following two, named using the notation of \cite{cas}. A real Banach space $X$ is said to be:
\begin{enumerate}[(a)]
  \item \emph{(GM) polyhedral} if $x^*(x)<1$ whenever $x\in S_X$ and $x^*\in [\ext(B_{X^*})]'$;
  \item \emph{(BD) polyhedral} if for each $x\in S_X$, $\sup\{x^*(x)\colon x^*\in\ext(B_{X^*})\setminus D(x)\}<1$.
\end{enumerate}
It is known that
\begin{equation*}\label{eq:GM=>BD}\tag{$\star\star\star$}
\text{(GM) polyhedral } \ \Longrightarrow\  \ \text{ (BD) polyhedral } \ \Longrightarrow\  \ \text{ polyhedral},
\end{equation*}
see \cite[Theorem 1]{DurierPapini} or \cite[Theorem 1.2]{FonfVesely}. The above implications does not reverse in general \cite{DurierPapini,FonfVesely}.

Our first observation is the following easy consequence of Theorem~\ref{theo:carapuntodaugapreduL1}, Remark~\ref{remark:algunasvalenengeneral}, and Proposition~\ref{prop:nicelyembedded}.

\begin{cor}\label{cor:equivalence-GMpolyhedral}
Let $X$ be a real Banach space.
\begin{enumerate}[(a)]
\item If $X$ is (GM) polyhedral, then the set of $\Delta$-points of $X$ is empty.
\item If $X$ is actually an $L_1$-predual, then $X$ is (GM) polyhedral if and only if the set of $\Delta$-points of $X$ is empty.
\item More in general, if $X$ is nicely embedded in some $C^b(S)$ space, then $X$ is (GM) polyhedral if and only if the set of $\Delta$-points of $X$ is empty.
\end{enumerate}
\end{cor}

Contrary to what it was stated during years in many papers, the implications in Eq.~\eqref{eq:GM=>BD} does not reverse for $L_1$-preduals, a result recently discovered \cite{cas}. Actually, (GM) polyhedrality and (BD) polyhedrality are not equivalent for $L_1$-preduals. This is exactly the ``breaking'' point for $L_1$-preduals \cite{cas}, as all versions of polyhedrality weaker than (BD) polyhedrality (including polyhedrality itself) are equivalent to (BD) polyhedrality for $L_1$-preduals. It is easy to see that examples of Banach spaces failing (BD) polyhedrality are $C(K)$ spaces and $C_0(L)$ spaces when the locally compact topological space $L$ has an accumulation point.

Inspired by Corollary \ref{cor:equivalence-GMpolyhedral}, one may wonder if the failure of (BD) polyhedrality for $L_1$-preduals can be characterised by some kind of ``massiveness'' of the set of $\Delta$-points. The example given in \cite{cas} to show that (GM) polyhedrality and (BD) polyhedrality are not equivalent, makes us think that a positive answer could be possible. Let us state the example here. Consider
$$
W=\left\{x\in c\colon \lim_n x(n)=\sum_{n=1}^\infty \frac{x(n)}{2^n}  \right\}.
$$
It is shown in \cite[Section 3]{cas} that $W$ is an $L_1$-predual (the operator $\phi\colon \ell_1 \longrightarrow W^*$ given by
$$
[\phi(y)](x)=\sum_{n=1}^{\infty} x(n)y(n) \ \ \text{for every $x\in W$ and every $y\in \ell_1$}
$$
is an onto isometry) and that $W$ is (BD) polyhedral. Besides, $\{e_n^*\}$ converges weakly-star to the functional $\left\{\frac{1}{2^n}\right\}$ in $S_{W^*}$. This shows that the constant function $1$ of $W$ is a $\Delta$-point by Theorem~\ref{theo:carapuntodaugapreduL1}, so $W$ is not (GM) polyhedral. Actually, it also follows from Theorem~\ref{theo:carapuntodaugapreduL1} that the only $\Delta$-points of $S_W$ are the constant function $1$ and its opposite, the constant function $-1$, since they are the only points of $S_W$ at which the functional $\left\{\frac{1}{2^n}\right\}$ attains its norm.

A property related to $\Delta$-points which implies some ``massiveness'' of the set of $\Delta$-points is the following one from \cite{ahlp}. Let $X$ be a Banach space, and let $\Delta_X$ be the set of $\Delta$-points of $S_X$. A Banach space $X$ is said to have the \emph{convex diametral local diameter two property} (\emph{convex-DLD2P} in short) if $B_X=\cconv(\Delta_X)$. This property is introduced in \cite{ahlp} as a property which is implied by the diametral local diameter two property or DLD2P (in our language, $\Delta_X=S_X$) and which implies that every slice of $B_X$ has diameter two \cite[Proposition 5.2]{ahlp}.
It is also shown in \cite{ahlp} that $C(K)$ spaces (with $K$ infinite) and M\"{u}ntz spaces on $[0,1]$ have the convex-DLD2P \cite[Proposition 5.3 and Theorem 5.7]{ahlp}, and that $c_0$ fails the convex-DLD2P \cite[Remark 5.5]{ahlp}. This shows, in particular, that the DLD2P, the convex-DLD2P, and the diameter two property of the slices are different properties even in the $L_1$-preduals ambient.

As a consequence of Theorem~\ref{theo:carapuntodaugapreduL1}, we get the following result on the convex-DLD2P.

\begin{cor}\label{cor:suff-convexDLD2P}
Let $X$ be an infinite-dimensional $L_1$-predual. If there is $e^*\in \ext(B_{X^*})$ which is the weak-star limit of a net $\{e_\lambda^*\}$ of pairwise linearly independent elements in $\ext(B_{X^*})$, then $X$ has the convex-DLD2P. In the real case, the previous condition can be replaced with $[\ext(B_{X^*})]'\cap \ext(B_{X^*}) \neq \emptyset$.
\end{cor}

\begin{proof}
As $e^*\in \ext(B_{X^*})$, it is known that the set $F(e^*)=\{x\in S_X\colon e^*(x)=1\}$ is not empty and, moreover, that
$$
B_X=\cconv\bigl(F(x)\bigr)
$$
(see \cite[Corollary 2.13]{SpearsBook}, for instance). Pick any $x\in F(e^*)$. As $\{|e^*_\lambda(x)|\}$ converges to $e^*(x)=1$ and the $\{e^*_\lambda\}$ are pairwise linearly independent, it follows from Theorem~\ref{theo:carapuntodaugapreduL1} that $x$ is a $\Delta$-point. Therefore, $X$ has the convex-DLD2P, as desired.
\end{proof}

Two particular cases of the above result are interesting. Item (a) extends the result on $C(K)$ spaces from \cite{ahlp} and item (c) provides new examples of $L_1$-preduals with the convex-DLD2P.

\begin{cor}\label{cor:suff-convexDLD2P-particularcases}
Let $X$ be an infinite-dimensional $L_1$-predual. Then, each of the following conditions implies the convex-DLD2P:
\begin{enumerate}
  \item[(a)] if $\ext(B_{X^*})$ is weak-star closed (in particular, if $X=C(K)$ for some compact space $K$);
  \item[(b)] if $S_X$ contains a $\Delta$-point at which the norm is smooth;
  \item[(c)] $X=C_0(L)$ for some locally compact space $L$ containing an accumulation point.
\end{enumerate}
\end{cor}

\begin{proof}
For (a), being $X$ infinite-dimensional, we may find a net $\{e_\lambda^*\}$ of pairwise linearly independent elements of $\ext(B_{X^*})$. Being $\{e_\lambda^*\colon \lambda \in \Lambda\}$ an infinite subset of the weakly-star compact subset $B_{X^*}$, it contains a weak-star limit point $e^*$. As $\ext(B_{X^*})$ is weakly-star closed, such a limit point must belong to $\ext(B_{X^*})$. But then, $e^*\in \ext(B_{X^*})$ is the weak-star limit of a net of pairwise linearly independent extreme points, so Corollary~\ref{cor:suff-convexDLD2P} gives the result.

For (b), pick $e^*\in \ext(B_{X^*})$ with $e^*(x)=1$. If the norm of $X$ is smooth at $x$, then $D(x)=\{e^*\}$. As $x$ is a $\Delta$-point, it follows from Theorem~\ref{theo:carapuntodaugapreduL1} that $D(x)\cap [\exposatx(B_{X^*})]'\neq \emptyset$, so there is a net of distint elements of $\exposatx(B_{X^*})$ which is weakly-star convergent to $e^*$. But distint elements of $\exposatx(B_{X^*})$ are clearly pairwise linearly independent. Then, Corollary~\ref{cor:suff-convexDLD2P} gives the result.
\end{proof}

Let us observe that part of the results in Corollaries \ref{cor:suff-convexDLD2P} and \ref{cor:suff-convexDLD2P-particularcases} can be obtained for nicely embedded Banach spaces, applying  Proposition~\ref{prop:nicelyembedded} instead of Theorem~\ref{theo:carapuntodaugapreduL1}. Only part of the language changes, so we only include a sketch of its proof.

\begin{cor}\label{coro:nicellyembedded-convexDLD2P}
Let $X$ be an infinite-dimensional Banach space which is nicely embedded into $C^b(S)$ for which $\{p_s\colon s\in S\}$ is (pairwise) linearly independent (in particular, if $X$ is a function algebra). Then, each of the following conditions implies that $X$ has the convex DLD2P:
\begin{enumerate}
\item[(a)] if there is a net $\{p_{s_\lambda}\}$ of distinct elements which is weak-star converging to some $p_{s_0}$;
\item[(b)] if $\{p_s\colon s\in S\}$ is weak-star closed;
\item[(c)] if $S_X$ contains a $\Delta$-point at which the norm is smooth.
\end{enumerate}
\end{cor}

\begin{proof}
(a) Since $\mathbb K p_{s_0}$ is $L$-embedded, it follows from \cite[Example 2.12.a]{SpearsBook} that $p_{s_0}$ is a spear element of $B_{X^*}$ (see \cite[Definition 2.1]{SpearsBook}), so \cite[Theorem~2.9]{SpearsBook} gives us that $A=\{x\in S_X\colon  p_{s_0}(x)=1\}$ is non-empty and that $\overline{\co}(A)=B_X$. The rest of the proof is completely analogous to the one of Corollary~\ref{cor:suff-convexDLD2P}, using Proposition~\ref{prop:nicelyembedded} instead of Theorem~\ref{theo:carapuntodaugapreduL1}.

(b) and (c) follows from the previous result in the same manner than it is done in the proof of Corollary~\ref{cor:suff-convexDLD2P-particularcases}.
\end{proof}

We do not know wether Corollary~\ref{cor:suff-convexDLD2P} or assertion (b) of Corollary~\ref{cor:suff-convexDLD2P-particularcases} characterises the convex-DLD2P for $L_1$-preduals. On the other hand, we also do not know if the convex-DLD2P for real $L_1$-preduals can be characterised in terms of the failure of some kind of polyhedrality. Let us emphasise the question.

\begin{prob}
We do not know if a real $L_1$-predual has the convex-DLD2P if and only if $X$ has fails to be (BD) polyhedral.
\end{prob}

\section{Polynomial Daugavet property}\label{sect:Polynomial}

Let us start by recalling the definition of the polynomial Daugavet property introduced in \cite{cgmm,cgmm2}.

\begin{definition}[\textrm{\cite{cgmm,cgmm2}}]\label{def:polynomialDaugavet}
A Banach space $X$ has the \emph{polynomial Daugavet property} if every weakly compact polynomial $P\in \mathcal{P}(X,X)$ satisfies the Daugavet equation
$$
\|\Id + P\|=1 + \|P\|.
$$
\end{definition}

Examples of Banach spaces with the polynomial Daugavet property include $C(K)$ spaces for perfect $K$, $L_1(\mu)$ and $L_\infty(\mu)$ for atomless $\mu$, and some generalisations of these examples, as non-atomic $C^*$-algebras, representable spaces, $C$-rich subspaces of $C(K)$ spaces, among others. We refer the reader to \cite{Bothelho-Santos,ChoiGarKimMaestre2014,cgmm,cgmm2,mmp,SantosCstar, santos2020} for more information and background. Let us comment that it is still unknown whether the Daugavet property always implies the polynomial Daugavet property. The following equivalent reformulation of the polynomial Daugavet property is well known and we will make use of it profusely, see \cite[Proposition 1.3 and Corollary 2.2]{cgmm} or \cite[Lemma 6.1]{cgmm2}:
\begin{quote}
$X$ has the polynomial Daugavet property if, and only if, given $x\in S_X$, $\varepsilon>0$, and a norm-one polynomial $p\in\mathcal P(X)$, there exists $y\in B_X$ and $\omega\in \T$ with $\re \omega P(y)>1-\varepsilon$ and $\Vert x+\omega y\Vert>2-\varepsilon$.
\end{quote}

Our main goal in this section is to use the results of Section~\ref{section:L1predDeltapoints} to prove the following extension of the fact that $C(K)$ spaces with the Daugavet property actually satisfy the polynomial Daugavet property.

\begin{theorem}\label{theorem:polyDPpreduL1}
The following spaces have the polynomial Daugavet property:
\begin{enumerate}[(a)]
  \item $L_1$-preduals with the Daugavet property;
  \item more in general, spaces nicely embedded in $C^b(\Omega)$ when $\Omega$ has no isolated points and for which $\{p_s\colon s\in S\}$ is (pairwise) linearly independent.
\end{enumerate}
\end{theorem}

The particular case of item (b) of the corollary above for uniform algebras whose Choquet boundaries have no isolated points is already known, see \cite[Theorem 2.7]{ChoiGarKimMaestre2014}.

The proof of Theorem~\ref{theorem:polyDPpreduL1} will follow directly from Theorem~\ref{theo:carapuntodaugapreduL1} and Proposition~\ref{prop:nicelyembedded} by using the following general result which extends \cite[Proposition 6.3]{cgmm2}.

\begin{prop}\label{prop:sufficient_polynomial_Daugavet}
Let $X$ be a Banach space. Suppose that given $x\in S_X$, $y\in B_X$, and $\omega\in \T$, there is a sequence $\{x_n^{**}\}$ in $B_{X^{**}}$ such that
$$
\limsup \|x + \omega x_n^{**}\|=2
$$
and that the linear operator from $c_0$ to $X^{**}$ defined by $e_n \longmapsto x_n^{**}-y$ for all $n\in \mathbb{N}$ is continuous. Then, $X$ has the polynomial Daugavet property.
\end{prop}

Observe that the only difference between the above result and \cite[Proposition~6.3]{cgmm2} is that, in the latter, the sequence $\{x_n^{**}\}$ has to belong to $X$.

\begin{proof}
Pick $x\in S_X$, $P\in\mathcal P(X)$ with $\|P\|=1$, and $\varepsilon>0$. Let us find an element $z\in B_X$ and $\omega\in \T$ such that $\re \omega P(z)>1-\varepsilon$ and that $\Vert x+\omega z\Vert>2-\varepsilon$, and then apply \cite[Proposition 1.3 and Corollary 2.2]{cgmm} to get that $X$ has the polynomial Daugavet property. To this end, pick $y\in B_X$ and $\omega\in \T$ such that $\re\omega P(y)>1-\varepsilon$ and let $\{x_n^{**}\}$ in $B_{X^{**}}$ the sequence given by the hypothesis. So, on the one hand, we have that
\begin{equation*}\label{eq:proof-polynomial1} \tag{$\star\star$}
\limsup \|x + \omega x_n^{**}\|=2.
\end{equation*}
Define a linear operator $T\colon c_0\longrightarrow X^{**}$ by $T(e_1)=y$ and $T(e_{n+1})=x_n^{**}-y$ for $n\in \mathbb{N}$. It follows from the hypothesis that $T$ is continuous. Therefore, $Q:=\widehat{P}\circ T\colon c_0 \longrightarrow \K$, where $\widehat P$ is the Aron-Berner extension of $P$, is a continuous polynomial on $c_0$. As $\{e_1+e_n\}$ converges weakly to $e_1$ in $c_0$, using the weak continuity of polynomials on bounded subsets of $c_0$ (see \cite[Proposition 1.59]{Dineen}), we get that $\{Q(e_1+e_n)\}\longrightarrow Q(e_1)$. In particular,
\begin{equation*}
\re \omega\widehat{P}(x_n^{**}) \longrightarrow \re \omega\widehat{P}(y)=\re\omega P(y)>1-\varepsilon.
\end{equation*}
This, together with Eq.~\eqref{eq:proof-polynomial1}, allows us to find $n\in \mathbb{N}$ such that
$$
\|x+\omega x_n^{**}\|>2-\varepsilon \quad \text{ and } \quad \re \omega \widehat{P}(x_n^{**})>1-\varepsilon.
$$
By \cite[Theorem 2]{dg} we can find a net $\{z_\alpha\}$ in $B_X$ converging to $x_n^{**}$ in the polynomial-star topology of $X^{**}$ (that is, $R(z_\alpha)\longrightarrow \widehat{R}(x_n^{**})$ for every polynomial $R\in \mathcal{P}(X)$) so, in particular, it also converges in the weak-star topology. Consequently, we can find $\alpha$ large enough so that
\begin{equation*}
\Vert x+\omega z_\alpha\Vert>2-\varepsilon \quad \text{ and } \quad \re\omega P(z_\alpha)>1-\varepsilon.
\qedhere
\end{equation*}
\end{proof}

\begin{proof}[Proof of Theorem~\ref{theorem:polyDPpreduL1}]
We only have to check that the hypotheses of Proposition~\ref{prop:sufficient_polynomial_Daugavet} are satisfied.
For $X$ being a $L_1$-predual with the Daugavet property, given $x\in S_X$, $y\in B_X$, and $\omega\in \T$, we just have to apply condition (6) of Theorem~\ref{theo:carapuntodaugapreduL1} for $-\bar{\omega}x\in S_X$ (which is a Daugavet point). For a nicely embedded space, as $S$ has no isolated points, the condition (d) of Proposition~\ref{prop:nicelyembedded} is clearly satisfies for $-\bar{\omega}x$, so item (e) of that proposition provides the proof that we are in the hypotheses of Proposition~\ref{prop:sufficient_polynomial_Daugavet}, as desired.
\end{proof}

A final result in this section will deal with the spaces of Lipschitz functions. Let us say that this result will not be used in Section \ref{section:WODP}, but we include it here as the same kind of arguments than the previous ones allows us to provide new examples of Banach spaces in which the Daugavet property and the polynomial Daugavet property are equivalent. Let us introduce briefly the necessary notation.  Given a metric space $M$ and a point $x\in M$, we will denote by $B(x,r)$ the closed ball centred at $x$ with radius $r$. Let $M$ be a metric space with a distinguished point $0 \in M$. The pair $(M,0)$ is commonly called a \emph{pointed metric space}. By an abuse of language, we will say only ``let $M$ be a pointed metric space'' and similar sentences.
The vector space of Lipschitz functions from $M$ to $\mathbb R$ will be denoted by $\justLip(M)$.
Given a Lipschitz function $f\in \justLip(M)$, we denote its Lipschitz constant by
\[ \Vert f\Vert_{L} = \sup\left\{ \frac{|f(x)-f(y)|}{d(x,y)}\colon x,y\in M,\, x\neq y\right\}. \]
This is a seminorm on $\justLip(M)$ which is a Banach space norm on the space $\Lip(M)\subseteq \justLip(M)$ of Lipschitz functions on $M$ vanishing at $0$.

\begin{prop}\label{prop:Lipschitz}
Let $M$ be a pointed complete metric space. If $\Lip(M)$ has the Daugavet property, then it actually has the polynomial Daugavet property.
\end{prop}

\begin{proof}
We will follow the lines of the proof of \cite[Proposition 3.3]{lr}. Pick $f,g\in S_{\Lip(M)}$. Notice that, by \cite[Proposition 3.4 and Theorem 3.5]{gpr}, for every $\varepsilon>0$ the set
$$
V_\varepsilon=\left\{x\in M\colon \inf\limits_{\beta>0}\Vert f_{|B(x,\beta)}\Vert>1-\varepsilon\right\}
$$
is infinite, so we can take a sequence $\{w_n\}$ of different points of $V_\frac{1}{n}$. An inductive argument allows to take $r_n>0$ small enough so that $d(w_n,w_m)\geq 2r_n$ holds for every $n> m$ and such that $\sum_{n=1}^\infty \frac{r_n}{8-r_n}<\infty$. By the property defining $w_n$, for every $n\in\mathbb N$ we can take a pair of different points $x_n,y_n\in B\left(w_n,\frac{r_n^2}{8}\right)$ such that $$f(x_n)-f(y_n)>(1-\frac{1}{n})d(x_n,y_n).$$ Now, we define $g_n\colon \bigl[M\setminus B(w_n,r_n)\bigr]\cup\{x_n,y_n\}\longrightarrow \mathbb R$ by $g_n(t)=g(t)$ if $t\neq x_n$ and $g_n(x_n)=g(y_n)+d(x_n,y_n)$. Let us estimate the norm of $g_n$. First, $g_n(x_n)-g_n(y_n)=d(x_n,y_n)$, so $\|g_n\|\geq 1$. Next, we only have to compare slopes of $g_n$ at $x_n$ and $z\notin B(w_n,r_n)$. Notice that $d(z,x_n)\geq r_n-\frac{r_n^2}{8}$ by the triangle inequality and, similarly, $d(z,y_n)\geq r_n-\frac{r_n^2}{8}$. Now,
\begin{align*}
\frac{\vert g_n(x_n)-g_n(z)\vert}{d(x_n,z)} & \leq \frac{\vert g(y_n)-g(z)\vert+d(x_n,y_n)}{d(x_n,z)} \leq \frac{d(y_n,z)+d(x_n,y_n)}{d(z,x_n)} \\
& \leq \frac{d(z,x_n)+d(x_n,y_n)}{d(z,x_n)}=1+\frac{2d(x_n,y_n)}{d(z,x_n)} \\ &\leq 1+\frac{\frac{r_n^2}{8}}{r_n-\frac{r_n^2}{8}}=1+\frac{r_n}{8-r_n}.
\end{align*}
By McShane's extension Theorem (see \cite[Theorem 1.33]{Weaver2}, for instance), we can extend $g_n$ to be defined in the whole of $M$ still satisfying
$$
1\leq \Vert g_n\Vert\leq 1+\frac{r_n}{8-r_n}.
$$
We clearly have that $\supp(g_n-g)\subseteq B(w_n,r_n)$. This implies that
$$
d(\supp(g_n-g),\supp(g_m-g))>0\ \text{ for every $n\neq m$}.
$$
Then, the arguments in the proof of \cite[Lemma~1.5]{ccgmr} implies that the operator $T\colon c_0\longrightarrow \Lip(M)$ given by $T(e_n):=g_n-g$ for every $n\in \mathbb{N}$ is continuous (even more, $\Vert T\Vert\leq 2$). Using that $\left\Vert \frac{g_n}{\Vert g_n\Vert}-g_n\right\Vert<\frac{r_n}{8-r_n}$ for every $n\in\mathbb N$, it is routine to prove that the operator $c_0\longrightarrow \Lip(M)$ given by $e_n\longmapsto \frac{g_n}{\Vert g_n\Vert}-g$ is bounded (actually, its norm is bounded by $2+\sum_{n=1}^\infty \frac{r_n}{8-r_n}$). On the other hand,
$$\left\Vert f+\frac{g_n}{\Vert g_n\Vert}\right\Vert\geq \frac{\bigl(f+\frac{g_n}{\Vert g_n\Vert}\bigr)(x_n)-\bigl(f+\frac{g_n}{\Vert g_n\Vert}\bigr)(y_n)}{d(x_n,y_n)}>1-\frac{1}{n}+\frac{1}{1+\frac{r_n}{8-r_n}}.$$
Finally, an application of Proposition \ref{prop:sufficient_polynomial_Daugavet} concludes that $\Lip(M)$ has the polynomial Daugavet property, as desired.
\end{proof}

\section{Weak operator Daugavet property and polynomial weak operator Daugavet property}\label{section:WODP}

Let us start the section by recalling the definition of the operator Daugavet property introduced in \cite{rtv} with the aim of providing a weaker version.

\begin{definition}[\textrm{\cite[Definition 4.1]{rtv}}]\label{def:ODP}
Let $X$ be a Banach space. We say that $X$ has the \emph{operator Daugavet property} (\emph{ODP} in short) if, given $x_1,\ldots x_n\in S_X$, $\varepsilon>0$, and a slice $S$ of $B_X$, there exists an element $x\in S$ such that, for every $x'\in B_X$ we can find an operator $T\colon X\longrightarrow X$ with $\Vert T\Vert\leq 1+\varepsilon$, $\Vert T(x_i)-x_i\Vert<\varepsilon$ for every $i\in \{1,\ldots,n\}$ and $T(x)=x'$.
\end{definition}

This property was introduced in the aforementioned paper \cite{rtv} as a sufficient condition for a pair of Banach spaces $X$ and $Y$ to get that $X\pten Y$ have the Daugavet property. Examples of spaces satisfying the ODP are $L_1$-preduals with the Daugavet property and $L_1(\mu,Y)$ when $\mu$ is atomless and $Y$ is arbitrary. Besides, this property is stable by finite $\ell_\infty$ sums.

Our strategy for proving the main results of this paper will be to consider the following weakening of the ODP.

\begin{definition}\label{def:WODP}
Let $X$ be a Banach space. We say that $X$ has the \emph{weak operator Daugavet property} (\emph{WODP} in short) if, given $x_1,\ldots x_n\in S_X$, $\varepsilon>0$, a slice $S$ of $B_X$ and $x'\in B_X$, we can find $x\in S$ and $T\colon X\longrightarrow X$ with $\Vert T\Vert\leq 1+\varepsilon$, $\Vert T(x_i)-x_i\Vert<\varepsilon$ for every $i\in \{1,\ldots,n\}$ and $\Vert T(x)-x'\Vert<\varepsilon$.
\end{definition}

The ODP clearly implies the WODP. Actually, the following result also holds.

\begin{remark}\label{remark:WODPimpliesDaugavet}
If $X$ is a Banach space with the WODP, then $X$ has the Daugavet property. Indeed, given $x\in S_X$, an slice $S$ of $B_X$, and $\varepsilon>0$, taking $x'=-x$ we can find, by the definition of WODP, an element $y\in S$ and an operator $T\colon X\longrightarrow X$ with $\Vert T\Vert\leq 1+\varepsilon$ and such that $\max\{\Vert T(x)-x\Vert,\Vert T(y)+x\Vert\}<\varepsilon$. It is not difficult to prove that $\Vert x+y\Vert\geq \frac{2-2\varepsilon}{1+\varepsilon}$.
\end{remark}

Our first interest in the WODP is that it is stable by projective tensor product, a result which improves the main ones of \cite{rtv}.

\begin{theorem}\label{theorem:WODP}
Let $X$ and $Y$ be two Banach spaces with the WODP. Then, $X\pten Y$ has the WODP.
\end{theorem}

We need the following technical lemma.

\begin{lemma}\label{lema:finislicesWODP}
Let $X$ be a Banach space with the WODP. Then, for every $x_1,\ldots, x_n\in S_X$, for every $y_1',\ldots, y_k'\in B_X$, every slices $S_1,\ldots, S_k$ of $B_X$ and every $\varepsilon>0$ we can find $y_j\in S_j$ for every $1\leq j\leq k$ an operator $T\colon X\longrightarrow X$ with $\Vert T\Vert\leq 1+\varepsilon$ satisfying that
$$
\Vert T(x_i)-x_i\Vert<\varepsilon\ \text{ for $1\leq i\leq n$} \quad \text{and} \quad \Vert T(y_j)-y_j'\Vert<\varepsilon\ \text{ for $1\leq j\leq k$}.
$$
\end{lemma}

\begin{proof}
Let us prove the result by induction on $k$. For the case $k=1$ there is nothing to prove. Now assume by induction hypothesis that the result holds for $k$, and let us prove the case $k+1$. To this end, pick $x_1,\ldots, x_n\in S_X$, $\varepsilon>0$, $S_1,\ldots, S_{k+1}$ slices of $B_X$ and $y_1',\ldots, y_{k+1}'\in B_X$, and let us find an operator $\phi$ witnessing the thesis of the lemma.

To this end, by the induction hypothesis, we can find $y_i\in S_i$ for $1\leq i\leq k$ and an operator $T\colon X\longrightarrow X$ with $\Vert T\Vert\leq 1+\varepsilon$ and such that
\begin{enumerate}
    \item $\Vert T(x_i)-x_i\Vert<\varepsilon$ for every $1\leq i\leq n$ and $\Vert T(y_{k+1}')-y_{k+1}'\Vert<\varepsilon$.
    \item $\Vert T(y_i)-y_i'\Vert<\varepsilon$ holds for every $1\leq i\leq k$.
\end{enumerate}

Now, by the definition of the WODP we can find $y_{k+1}\in S_{k+1}$ and an operator $G\colon X\longrightarrow X$ with $\Vert G\Vert\leq 1+\varepsilon$ and such that
\begin{itemize}
    \item[(3)] $\Vert G(x_i)-x_i\Vert<\varepsilon$ for $1\leq i\leq n$ and $\Vert G(y_j)-y_j\Vert<\varepsilon$ for $1\leq j\leq k$.
    \item[(4)] $\Vert G(y_{k+1})-y_{k+1}'\Vert<\varepsilon$.
\end{itemize}
Define $\phi:=T\circ G\colon X\longrightarrow X$ and let us prove that $\phi$ satisfies our purposes. First, $\Vert \phi\Vert\leq (1+\varepsilon)^2$.
Next, given $1\leq i\leq n$ we have
\[\begin{split}
    \Vert \phi(x_i)-x_i\Vert &=\Vert T(G(x_i))-T(x_i)+T(x_i)-x_i\Vert \\ & \leq \Vert T(G(x_i)-x_i)\Vert+\Vert T(x_i)-x_i\Vert\\
    & \leq \Vert T\Vert \Vert G(x_i)-x_i\Vert+\varepsilon\\
    & < (1+\varepsilon)\varepsilon+\varepsilon=(2+\varepsilon)\varepsilon
\end{split}\]
just combining (1) and (3). Moreover, given $i\in\{1,\ldots, k\}$, we obtain
\[\begin{split}
    \Vert \phi(y_i)-y_i'\Vert & =\Vert T(G(y_i))-T(y_i)+T(y_i)-y_i'\Vert\\ & \leq   \Vert T(G(y_i)-y_i)\Vert+\Vert T(y_i)-y_i'\Vert\\
    & \leq \Vert T\Vert \Vert G(y_i)-y_i\Vert+\varepsilon\\
    & <(1+\varepsilon)\varepsilon+\varepsilon=(2+\varepsilon)\varepsilon
\end{split}
\]
by combining (2) and (3). Finally,
\[\begin{split}
    \Vert \phi(y_{k+1})-y_{k+1}'\Vert& =\Vert T(G(y_{k+1}))-T(y_{k+1}')+T(y_{k+1}')-y_{k+1}'\Vert\\
    & \leq \Vert T(G(y_{k+1})-y_{k+1}')\Vert+\Vert T(y_{k+1}')-y_{k+1}'\Vert\\
    & <\Vert T\Vert \Vert G(y_{k+1})-y_{k+1}'\Vert+\varepsilon\\
    & <(1+\varepsilon)\varepsilon+\varepsilon=(2+\varepsilon)\varepsilon
\end{split}
\]
by combining (1) and (4). This proves, up to making a choice of a smaller $\varepsilon$, that $\phi$ is our desired operator.
\end{proof}

We are now ready to give the pending proof.

\begin{proof}[Proof of Theorem~\ref{theorem:WODP}] Let $Z:=X\pten Y$.
Fix $z_1,\ldots, z_n\in B_{Z}$, $\varepsilon>0$, $z'\in B_Z$, and a slice $S=S(B_Z,B,\alpha)$ for certain norm-one bilinear form $B\colon X\times Y\longrightarrow \mathbb K$.

By a density argument, we can assume with no loss of generality that $$z_i=\sum_{j=1}^{n_i} \lambda_{ij} a_{ij}\otimes b_{ij}\in \conv(S_X\otimes S_Y)\qquad i\in\{1,\ldots,n\}$$ and, in a similar way, that $z'=\sum_{k=1}^t \mu_k x_k'\otimes y_k'\in \conv(S_X\otimes S_Y)$.

Take $u_0\otimes v_0\in S$ with $u_0\in B_X$ and $v_0\in B_Y$, which means $\re B(u_0,v_0)>1-\alpha$ or, equivalently, that $u_0\in S':=\{z\in B_X\colon \re B(z,v_0)>1-\alpha\}$, which is a slice of $B_X$. By Lemma~\ref{lema:finislicesWODP}, for every $1\leq k\leq t$ we can find an element $x_k\in S'$ (which implies that $x_k\otimes v_0\in S$) and an operator $T\colon X\longrightarrow X$ with $\Vert T\Vert\leq 1+\varepsilon$, satisfying that
$$
\Vert T(a_{ij})-a_{ij}\Vert<\varepsilon \ \text{ for every $i,j$} \quad \text{and} \quad \Vert T(x_k)-x_k'\Vert<\varepsilon \ \text{ for every $k$}.
$$
Notice that $v_0\in S_k:=\{z\in B_Y\colon \re B(x_k,z)>1-\alpha\}$ for every $k\in\{1,\ldots, t\}$. Again, by the previous lemma, for every $k\in\{1,\ldots, t\}$ we can find $y_k\in S_k$ (which means that $x_k\otimes y_k\in S$) and an operator $U\colon Y\longrightarrow Y$ with $\Vert U\Vert\leq 1+\varepsilon$ satisfying that
$$
\Vert U(b_{ij})-b_{ij}\Vert<\varepsilon\ \text{ for every $i,j$} \quad \text{and} \quad \Vert U(y_k)-y_k'\Vert<\varepsilon\ \text{for $1\leq k\leq t$}.
$$

Now, define $z:=\sum_{k=1}^t \mu_k x_k\otimes y_k$. Notice that $z\in S$ since
$$
\re B(z)=\sum_{k=1}^t \mu_k  \re B(x_k,y_k)>(1-\alpha)\sum_{k=1}^t \mu_k=1-\alpha.
$$
Finally define $\phi:=T\otimes U\colon Z\longrightarrow Z$. By \cite[Proposition 2.3]{rya}, $\Vert \phi\Vert=\Vert T\Vert \Vert U\Vert\leq (1+\varepsilon)^2$. On the other hand, given $1\leq i\leq n$, we get
\[
\begin{split}
    \Vert \phi(z_i)-z_i\Vert& =\left\Vert \sum_{j=1}^{n_i}\lambda_{ij}( T(a_{ij})\otimes T(b_{ij})-a_{ij}\otimes b_{ij})\right\Vert\\
    & \leq \sum_{j=1}^{n_i}\lambda_{ij} \Vert T(a_{ij})\otimes T(b_{ij})-T(a_{ij})\otimes b_{ij}+T(a_{ij})\otimes b_{ij}-a_{ij}\otimes b_{ij}\Vert\\
    & \leq \sum_{j=1}^{n_i}\lambda_{ij} (\Vert T(a_{ij})\Vert \Vert T(b_{ij})-b_{ij}\Vert+\Vert T(a_{ij})-a_{ij}\Vert \Vert b_{ij}\Vert)\\
    & <\sum_{j=1}^{n_i}\lambda_{ij} ((1+\varepsilon)\varepsilon+\varepsilon) =(2+\varepsilon)\varepsilon\sum_{j=1}^{n_i}\lambda_{ij}=(2+\varepsilon)\varepsilon.
\end{split}
\]
Similar estimates to the previous ones prove that $\Vert \phi(z)-z'\Vert<(2+\varepsilon)\varepsilon$.
\end{proof}

Our next goal is to introduce the polynomial WODP. We need some notation. Given $x_1,\ldots, x_n\in S_X, \varepsilon>0$, and $x'\in B_X$, write
$$
\OF(x_1,\ldots, x_n;x',\varepsilon):=\left\{y\in B_X\colon \begin{array}{c}
 \text{ there exits }T\colon X\longrightarrow X\\
\Vert T\Vert\leq 1+\varepsilon,\ \Vert T(y)-x'\Vert<\varepsilon,\\
\Vert T(x_i)-x_i\Vert<\varepsilon\ \forall i\in\{1,\ldots, n\}
\end{array}\right\}.
$$
Notice that a Banach space $X$ has the WODP if, and only if, all the sets of the form $\OF(x_1,\ldots, x_n;x',\varepsilon)$ are norming for $X^*$, that is, if and only if
$$
B_X=\cconv\bigl(\OF(x_1,\ldots, x_n;x',\varepsilon)\bigr)
$$
regardless of $x_1,\ldots, x_n, x',\varepsilon$. Actually, the sets are even more massive in this case, as the following result exhibits.

\begin{lemma}
Let $X$ be a Banach space with the WODP. Then, for every $x_1,\ldots, x_n\in S_X$, $x'\in B_X$, and $\varepsilon>0$ the set $\OF(x_1,\ldots, x_n;x',\varepsilon)$ intersects any convex combination of slices of $B_X$. In particular, the set is weakly dense.
\end{lemma}

\begin{proof}
Pick $C=\sum_{k=1}^t\lambda_k S_k$ to be a convex combination of slices of $B_X$. Pick $x_1,\ldots, x_n\in S_X$, $\varepsilon>0$ and $x'\in B_X$. Lemma~\ref{lema:finislicesWODP} allows us to find $y_k\in S_k$ for every $k\in\{1,\ldots,t\}$ and an operator $T\colon X\longrightarrow X$ satisfying that $\Vert T\Vert\leq 1+\varepsilon$, $\Vert T(x_i)-x_i\Vert<\varepsilon$ for $i\in\{1,\ldots,n\}$, and $\Vert T(y_k)-x'\Vert<\varepsilon$ for $k\in\{1,\ldots,t\}$. Now,
\[\begin{split}\left\Vert T\left(\sum_{k=1}^t \lambda_k y_k\right)-x'\right\Vert & = \left\Vert \sum_{k=1}^t \lambda_k T(y_k)-\sum_{k=1}^t \lambda_k x'\right\Vert \\ & \leq \sum_{k=1}^t \lambda_k \Vert T(y_k)-x'\Vert <\varepsilon\sum_{k=1}^t\lambda_k=\varepsilon.
\end{split}
\]
This implies that $\OF(x_1,\ldots, x_n;x',\varepsilon)\cap C\neq \emptyset$, as desired. Finally, the weak denseness of $\OF(x_1,\ldots, x_n;x',\varepsilon)$ follows since every non-empty weakly open subset of $B_X$ contains a convex combination of slices of $B_X$ by Bourgain's Lemma (see \cite[Lemma II.1]{ggms} for instance).
\end{proof}

Let us now consider the definition of polynomial WODP which is an stronger version of the WODP.

\begin{definition}\label{defi:polywodp}
Let $X$ be a Banach space. We say that $X$ has the \emph{polynomial weak operator Daugavet property} (\emph{polynomial WODP} in short) if, for every $P\in {\mathcal P(X)}$ with $\|P\|=1$, every $x_1,\ldots, x_n\in S_X$, $x'\in B_X$, $\alpha>0$ and $\varepsilon>0$, there exists $y\in \OF(x_1,\ldots, x_n;x',\varepsilon)$ and $\omega\in \T$ with $\re\omega P(y)>1-\alpha$.
\end{definition}

\begin{remark}
The polynomial WODP implies the WODP and the polynomial Daugavet property. Indeed, the first assertion is immediate since bounded linear functionals are in particular continuous polynomials. The second assertion follows with similar ideas behind the implication WODP $\Rightarrow$ Daugavet property in Remark~\ref{remark:WODPimpliesDaugavet}.
\end{remark}

Now it is time to exhibit examples of Banach spaces with the polynomial WODP, which in turn provides examples of Banach spaces with the WODP.

The first family of examples is the one of $L_1$-preduals with the Daugavet property.

\begin{proposition}\label{prop:WODPpolinomial-predual}
If $X$ is an $L_1$-predual with the Daugavet property, then $X$ has the polynomial WODP.
\end{proposition}

We will need a technical result which follows from \cite[Proposition~2.3]{santos2020} in the real case and which can be adapted also for the complex case.

\begin{lemma}\label{lemma:Santos}
Let $X$ be a Banach space with the polynomial Daugavet property. Then, given a finite-dimensional subspace $F$ of $X$, a norm-one polynomial $P\in \mathcal P(X)$, and $\varepsilon>0$, $\alpha>0$, there exists a norm-one polynomial $Q\in\mathcal P(X)$ and $\alpha_1>0$ satisfying that:
\begin{enumerate}[(a)]
  \item the set $\{z\in B_X\colon \vert Q(z)\vert >1-\alpha_1\}$ is contained in $\{z\in B_X\colon \vert P(z)\vert >1-\alpha\}$;
  \item the inequality $$\Vert e+ \lambda x\Vert>(1-\varepsilon)\bigl(\Vert e\Vert+\vert \lambda\vert\bigr)$$ holds for every $e\in F$, every $\lambda\in \mathbb K$, and every $x\in \{z\in B_X\colon \vert Q(z)\vert >1-\alpha_1\}$.
\end{enumerate}
\end{lemma}

\begin{proof}
Let us start proving the following:

\emph{Claim.}
For every $y_0\in S_X$, every norm-one polynomial $P\in \mathcal P(X)$, and $\alpha>0$, $\varepsilon>0$, there exist a norm-one polynomial $Q\in\mathcal P(X)$ and $\alpha'>0$ such that the set $\{z\in B_X\colon \vert Q(z)\vert>1-\alpha'\}$ is contained in $\{z\in B_X\colon \vert P(z)\vert>1-\alpha\}$ and satisfies that the inequality
$$
\left \Vert y_0+ \tfrac{\overline{P(x)}}{|P(x)|} x\right \Vert>2-\varepsilon
$$
holds for every $x\in B_X$ satisfying that $|Q(x)|>1-\alpha'$.

Notice that an inductive argument allows us to replace $y_0$ with any finite subset $\{y_1,\ldots, y_n\}\subseteq S_X$.

Indeed, define $\phi\colon X\longrightarrow X$ by $\phi(z):=P(z)y_0$ for every $z\in X$, which is a rank-one norm-one polynomial. Since $X$ has the polynomial Daugavet property, it follows that $2=\Vert \Id+\phi\Vert=\Vert \Id^*+\phi^*\Vert$, where $\phi^*\colon X^*\longrightarrow \mathcal P(X)$ is defined by $\phi^*(y^*):=y^*\circ \phi$ (see \cite{santos2020} for background). So, taking  $0<\varepsilon'<\min\{\varepsilon,\alpha\}>0$, we can find $y^*\in S_{X^*}$ such that
$$
\Vert y^*+y^*\circ \phi\Vert>2-\varepsilon'.
$$
Next, define
$$
Q:=\frac{y^*+y^*\circ \phi}{\Vert y^*+y^*\circ \phi\Vert} \quad \text{ and } \quad \alpha':=1-\frac{2-\varepsilon'}{\Vert y^*+y^*\circ \phi\Vert}.
$$
Let us prove that $Q$ and $\alpha'$ satisfies the required properties, following a similar argument to that of \cite[Theorem 2.2]{santos2020}. Pick $x\in B_X$ such that $\vert Q(x)\vert>1-\alpha'$. Then,
$$
2-\varepsilon'<\vert y^*(x)+P(x)y^*(y_0)\vert\leq 1+\vert P(x)y^*(y_0)\vert \leq 1+\vert P(x)\vert,
$$
from where it follows that $\vert P(x)\vert>1-\varepsilon' > 1-\alpha$. Moreover, such $x$ also satisfies that
$$
\Vert x+P(x)y_0\Vert \geq \vert y^*(x)+P(x)y^*(y_0)\vert >2-\varepsilon'
$$
and so,
$$
\left \Vert y_0+ \tfrac{\overline{P(x)}}{|P(x)|} x\right \Vert = \tfrac{1}{|P(x)|} \bigl\| x + P(x) y_0\bigr\| >2-\varepsilon' >  2 - \varepsilon,
$$
finishing the proof of the claim.

Now, if we take a $\delta$-net $A$ of $S_F$, for $\delta>0$ small enough, a standard argument (see the proofs of \cite[Proposition 2.3]{santos2020} or of \cite[Lemma~II.1.1]{lucking}) provide a norm-one polynomial $Q$ and $\alpha_1$ such that (a) is satisfied and so that the inequality
$$
\left\Vert y+\tfrac{\overline{P(x)}}{|P(x)|}x\right\Vert>(1-\varepsilon)(\Vert y\Vert+1)
$$
holds for every $x\in\{z\in B_X\colon \vert Q(z)\vert>1-\alpha_1\}$ and every $y\in F$. Being $Y$ a subspace, by just rotating $y$, we actually have that
$$
\left\Vert y+ x\right\Vert>(1-\varepsilon)(\Vert y\Vert+1).
$$
Using again that $Y$ is a subspace, we routinely get the desired inequality in (b).
\end{proof}

We are now ready to provide the pending proof.

\begin{proof}[Proof of Proposition~\ref{prop:WODPpolinomial-predual}]
Fix $x_1,\ldots, x_n\in S_X$, $x'\in B_X$, $\varepsilon>0$, $\alpha>0$, and $P\in \mathcal P(X)$ with $\|P\|=1$.
By Theorem~\ref{theorem:polyDPpreduL1} we get that $X$ has the polynomial Daugavet property. Hence, by Lemma~\ref{lemma:Santos} there exists an element $y\in B_X$ and $\omega\in \mathbb T$ with $\re \omega P(y)>1-\alpha$ and such that, denoting $E:=\lin\{x_1,\ldots, x_n\}$, we have that
$$
\Vert e+\lambda y\Vert>(1-\varepsilon)(\Vert e\Vert+\vert\lambda\vert)
$$
holds for every $e\in E$ and every $\lambda\in\mathbb K$. Define $T\colon E\oplus\mathbb K y\longrightarrow X$ by the equation
$$
T(e+\lambda y):=e+\lambda x'.
$$
Notice that
$$
\Vert T(e+\lambda y)\Vert =\Vert e+\lambda x'\Vert\leq \Vert e\Vert+\vert\lambda\vert\leq \frac{1}{1-\varepsilon}\Vert e+\lambda y\Vert,
$$
so $\Vert T\Vert\leq \frac{1}{1-\varepsilon}$. Since $X$ is an $L_1$-predual, $T$ can be extended to the whole of $X$ (still denoted by $T$) with norm $\Vert T\Vert\leq \frac{1+\varepsilon}{1-\varepsilon}$ (the real case follows from \cite[Theorem 6.1]{linds} and the complex case from \cite{hustad}, see \cite[p.\ 3]{lima}).

Since $T(x_i)=x_i$ and $T(y)=x'$, it follows that $y\in \OF\left(x_1,\ldots, x_n;y,\frac{2\varepsilon}{1-\varepsilon}\right)$. This, the arbitrariness of $\varepsilon>0$ and the fact that $\re \omega P(y)>1-\alpha$, show that $X$ has the polynomial WODP.
\end{proof}

The second family that we would like to present is the one of vector-valued $L_1$ spaces.

\begin{proposition}\label{prop:WODPpolinomial-L1vectorvalued}
Let $\mu$ be an atomless $\sigma$-finite positive measure and let $Y$ be a Banach space. Then, $L_1(\mu,Y)$ has the polynomial WODP.
\end{proposition}

\begin{proof}
Fix $x_1,\ldots, x_n\in S_{L_1(\mu,Y)}$, $x'\in B_{L_1(\mu,Y)}$, $\varepsilon>0$, $\alpha>0$, and $P\in \mathcal P(X)$ with $\|P\|=1$. From the finiteness of $\{x_1,\ldots, x_n\}$ and the fact that $\mu$ is atomless, we may find $\delta>0$ satisfying that
$$
A\in \Sigma,\ \mu(A)<\delta\ \ \Longrightarrow \ \ \int_A \Vert x_i\Vert<\frac{\varepsilon}{2}
$$
By the proof of \cite[Theorem~3.3]{mmp} there are $g\in S_{L_1(\mu,Y)}$ and $\omega\in \T$ satisfying
\begin{equation*}
\mu(\supp(g))<\delta \quad \text{ and } \quad \re \omega P(g) >1-\alpha.
\end{equation*}
Write $B:=\supp(g)$. As $L_\infty(\mu,Y^*)$ is norming for $L_1(\mu,Y)$ (because $L_1(\mu)^*=L_\infty(\mu)$ and simple functions are dense in $L_1(\mu,Y)$), we can find $h\in S_{L_\infty(\mu,Y^*)}$ such that
$$
\supp(h)\subseteq B\quad \text{and} \quad \re \langle h,g\rangle=\re \int_B \langle h(t),g(t)\rangle\, d\mu(t)>1-\varepsilon.
$$
By using again the denseness of simple functions, and taking into account that $\int_{B}\vert x_i\vert<\frac{\varepsilon}{2}$, we can find pairwise disjoint sets $C_1,\ldots, C_t\in \Sigma$ with positive and finite measure, all of them included in $\Omega\setminus B$, and $a_i^j\in Y$, $i\in\{1,\ldots, n\}$, $j\in \{1,\ldots, t\}$, such that $x_i':=\sum_{k=1}^t a_i^k \chi_{C_k}$ satisfies
$$
\Vert x_i-x_i'\Vert<\frac{\varepsilon}{2}.
$$
Define now $T\colon L_1(\mu,Y)\longrightarrow L_1(\mu,Y)$ by the equation
$$
T(f):=\sum_{k=1}^t\left(\frac{1}{\mu(C_k)}\int_{C_k} f\ d\mu \right)\chi_{C_k}+ \left(\int_{B} \langle h(t),f(t) \rangle\, d\mu(t)\right) x'.
$$
It is not difficult to see that $\Vert T\Vert\leq 1$ and that $T(x_i')=x_i'$, so
$$
\Vert T(x_i)-x_i\Vert\leq \Vert T(x_i-x_i')\Vert+\Vert x_i'-x_i\Vert<\varepsilon.
$$
Also, since $C_i\cap B=\emptyset$ and $\supp(g)=B$, we get
$$
\Vert T(g)-x'\Vert\leq \left\vert 1-\int_B \langle h(t),g(t)\rangle\, d\mu(t)\right\vert \Vert x'\Vert<\sqrt{2\varepsilon}.
$$
This concludes the proof.
\end{proof}

Now we are ready to establish the following result which, together with Propositions \ref{prop:WODPpolinomial-predual} and \ref{prop:WODPpolinomial-L1vectorvalued}, provides the promised proof of Theorem~\ref{theo:tensorsymejemplos}.

\begin{theorem}\label{theo:polynomialWODPimpliesptensNWODP}
Let $X$ be a Banach space with the polynomial WODP and let $N\in \mathbb{N}$. Then, $\ptensN X$ has the WODP and so, the Daugavet property.
\end{theorem}

We will need the following result which can be proved by induction with a similar argument to the one of Lemma~\ref{lema:finislicesWODP}.

\begin{lemma}\label{lema:finipolyslices}
Let $X$ be a Banach space with the polynomial WODP. Then, for every $x_1,\ldots, x_n\in S_X$, every $\varepsilon>0$, every polynomials $P_1,\ldots, P_k\in S_{\mathcal P(X)}$ and every $x_1',\ldots, x_k'\in B_X$ we can find $y_j\in B_X$ and $\omega_j\in \T$ for $1\leq j\leq k$ and an operator $T\colon X\longrightarrow X$ with $\Vert T\Vert\leq 1+\varepsilon$ satisfying  that $$
\re \omega_j P(y_j)>1-\alpha\ \text{ and } \ \Vert T(y_j)-x_j'\Vert<\varepsilon \ \text{ for every $1\leq j\leq k$},
$$
and that
$$
\Vert T(x_i)-x_i\Vert<\varepsilon \ \text{ for every $1\leq i\leq n$}.
$$
\end{lemma}

We divide the proof in two cases: first, when either $\K=\C$ or $\K=\R$ and $N$ is odd; second, when $\K=\R$ and $N$ is even. Observe that the difference is whether we may find $N$-roots of every scalar or not.

Let us start with the first case: either $\K=\C$ or $\K=\R$ and $N$ is odd.

\begin{proof}[Proof of Theorem~\ref{theo:polynomialWODPimpliesptensNWODP} for either $\K=\C$ or $\K=\R$ and $N$ odd]$ $\newline
Let $Y:=\ptensN X$. Pick $z_1,\ldots, z_n\in B_Y$, $\varepsilon>0$, $z'\in B_Y$ and $S:=S(B_Y,P,\alpha)$, for a certain $P\in S_{\mathcal P(^N X)}$. By a density argument, we can assume that $$z_i:=\sum_{j=1}^{n_i} \lambda_{ij} x_{ij}^N \in \conv(S_X^N)\ \ \text{ and }\ \ z'=\sum_{k=1}^t \mu_k (x_k')^N\in \conv(S_X^N).
$$
By Lemma~\ref{lema:finipolyslices} we can find $y_k\in B_X$ with $\re P(y_k)>1-\alpha$ and $T\colon X\longrightarrow X$ with $\Vert T\Vert\leq 1+\varepsilon$ and such that $\Vert T(x_{ij})-x_{ij}\Vert<\varepsilon$ for every $i,j$ and such that $\Vert T(y_k)-x_k'\Vert<\varepsilon$ for every $k\in\{1,\ldots, t\}$. Define $z:=\sum_{k=1}^t \mu_k y_k^N\in B_Y$. First of all, notice that $z\in S$. Indeed,
$$\re P(z)=\sum_{k=1}^t \mu_k \re P(y_k^N)=\sum_{k=1}^t \mu_k \re P(y_k)>(1-\alpha)\sum_{k=1}^t \mu_k=1-\alpha.$$
Now, define $\phi:=T^N\colon Y\longrightarrow Y$ by the equation
$$\phi(a^N):=T(a)^N.$$
By \cite[P.10, Proposition. (6)]{flo} we get that $\Vert \phi\Vert=\Vert T\Vert^N<(1+\varepsilon)^N$. Let us estimate $\Vert \phi(z_i)-z_i\Vert$. Indeed, given $i\in\{1,\ldots, n\}$ we get
\[\begin{split}
    \Vert \phi(z_i)-z_i\Vert&=\left\Vert \sum_{k=1}^{N_i}\lambda_{ij}(T(x_{ij})^N-x_{ij}^N)\right\Vert \leq \sum_{j=1}^{N_i}\lambda_{ij}\Vert T(x_{ij})^N-x_{ij}^N\Vert\\
    & \leq \frac{N^N}{N!}\sum_{j=1}^{n_i}\lambda_{ij}\Vert T(x_{ij})^N-x_{ij}^N\Vert_{X\pten X\pten\ldots\pten X}.
\end{split}\]
where in the last inequality we have used the polarization constant (see \cite[P.11 Subsection 2.3]{flo}).
Now, fix $j\in\{1,\ldots, n_i\}$. Then, in $X\pten X\pten\cdots\pten X$, we get the following equality
\begin{align*}
T(x_{ij})^N-x_{ij}^N & =\sum_{k=1}^N T(x_{ij})^{N-k+1}x_{ij}^{k-1}-T(x_{ij})^{N-k}x_{ij}^k \\
& =\sum_{k=1}^N T(x_{ij})^{N-k}\otimes (T(x_{ij})-x_{ij})\otimes x_{ij}^k.
\end{align*}
So,
\[\begin{split}\left\Vert T(x_{ij})^N-x_{ij}^N\right\Vert&\leq \sum_{k=1}^N \Vert T(x_{ij})\Vert^{N-k} \Vert T(x_i)-x_i\Vert\Vert x_{ij}\Vert^k <\sum_{k=1}^N (1+\varepsilon)^{N-k}\varepsilon\\
& <\varepsilon \sum_{k=0}^N(1+\varepsilon)^k =\varepsilon\frac{1-(1+\varepsilon)^{N+1}}{-\varepsilon} =(1+\varepsilon)^{N+1}-1.
\end{split}\]
Puting all together, we get
$$
\Vert \phi(z_i)-z_i\Vert\leq \frac{N^N}{N!}\sum_{j=1}^{n_i} \lambda_{ij}((1+\varepsilon)^{N+1}-1)=\frac{N^N}{N!}((1+\varepsilon)^{N+1}-1).
$$
Similar estimates to the above ones prove also that
$$\Vert \phi(z)-z'\Vert<\frac{N^N}{N!}((1+\varepsilon)^{N+1}-1).$$
The arbitrariness of $\varepsilon>0$ gives that $Y$ has the WODP, as desired.
\end{proof}

For the case of an even number $N$ and $\mathbb K=\mathbb R$, the proof will be similar but a bit more delicate. Notice that, given a polynomial $P\in S_{\mathcal P(^N X)}$, it is not true, in contrast with the odd case, that $\sup\limits_{x\in S_X} P(x)=1$ and $\inf\limits_{x\in S_X} P(x)=-1$ but we can only guarantee that one of those condition meets (in other words, the set $S_X^N$ is not balanced in $\ptensN X$). This induces a technical difficulty because, given a slice $S=S(B_{\ptensN X},P,\alpha)$ and given $\xi x^N\in S$, for $x\in B_X$ and $\xi\in \{-1,1\}$, we will not be able to determine the sign of $\xi$. This difficultly will be overcome by taking a smaller slice using the following technical result.

\begin{lemma}\label{lema:polypar}
Let $X$ be a real Banach space and let $N$ be an even number. Take $P\in S_{\mathcal P(^N X)}$ and assume that $\sup\limits_{x\in S_X} P(x)=1$. Then, for every $\alpha>0$ there exists a polynomial $Q\in B_{\mathcal P(^N X)}$ with the following properties:
\begin{enumerate}
    \item $\Vert Q\Vert>1-\frac{\alpha}{2}$.
    \item If $\xi\in\{-1,1\}$ and $y\in B_X$ are so that $\xi Q(y)>1-\frac{\alpha}{2}$ then $\xi=1$.
    \item If $Q(y)>1-\frac{\alpha}{2}$ for $y\in B_X$ then $P(y)>1-\alpha$.
\end{enumerate}
\end{lemma}

Note that if $\inf\limits_{x\in S_X} P(x)=-1$, an analogous statement holds making appropriate change of sings and order in (2) and (3).

\begin{proof}
Pick $x_0\in S_X$ such that $P(x_0)>1-\alpha$. Pick $x^*\in S_{X^*}$ so that $x^*(x_0)=1$ and define $Q:=\frac{P+(x^*)^N}{2}$. Notice that $Q\in B_{\mathcal P(^N X)}$ and that
$$Q(x_0)=\frac{P(x_0)+1}{2}>\frac{2-\alpha}{2}=1-\frac{\alpha}{2}$$
which proves (1). Moreover, if $\xi\in \{-1,1\}$ and $y\in B_X$ satisfies that $\xi Q(y)>1-\frac{\alpha}{2}$, then
$$
1-\frac{\alpha}{2}<\frac{\xi P(y)+\xi x^*(y)^N}{2}\leq \frac{1+\xi x^*(y)^N}{2}.
$$
It follows that $2-\alpha<1+\xi x^*(y)^N$, so $\xi x^*(y)^N>1-\alpha$.
Since $x^*(y)^N>1-\alpha$ because $N$ is even, we get that $\xi=1$ which proves (2). Finally, we get (3) by a simple convexity argument similar to the previously exposed.
\end{proof}

Now we are able to prove the remaining case.

\begin{proof}[Proof of Theorem~\ref{theo:polynomialWODPimpliesptensNWODP} for $\K=\R$ and $N$ even]$ $\newline
Let $Y:=\ptensN X$. Pick $z_1,\ldots, z_n\in B_Y$, $\varepsilon>0$, $z'\in B_Y$ and $S:=S(B_Y,P,\alpha)$, for a certain $P\in S_{\mathcal P(^N X)}$. By a density argument, for every $i$ we can assume that $z_i:=\sum_{j=1}^{n_i} \lambda_{ij} x_{ij}^N \in \aconv(S_X^N)$ with $\sum_{j=1}^{n_i}|\lambda_{ij}|=1$, and also that $z'=\sum_{k=1}^t \mu_k (x_k')^N\in \aconv(S_X^N)$ with $\sum_{k=1}^t\vert \mu_k\vert=1$. Pick a polynomial $Q\in B_{\mathcal P(^N X)}$ satisfying the thesis of Lemma~\ref{lema:polypar} and notice that $S(B_Y,Q,\frac{\alpha}{2})\subseteq S$.

By Lemma~\ref{lema:finipolyslices} we can find $y_k\in B_X$ with $Q(y_k)>1-\frac{\alpha}{2}$ (and so $y_k^N\in S$) and $T\colon X\longrightarrow X$ with $\Vert T\Vert\leq 1+\varepsilon$ and such that $\Vert T(x_{ij})-x_{ij}\Vert<\varepsilon$ for every $i,j$ and such that
$$\Vert T(\sign(\mu_k)y_k)-x_k'\Vert=\Vert T(y_k)-\sign(\mu_k)x_k'\Vert<\varepsilon$$
for every $k\in\{1,\ldots, t\}$. Define $z:=\sum_{k=1}^t \vert \mu_k\vert y_k^N\in B_Y$. First of all, notice that $z\in S$. Indeed,
$$Q(z)=\sum_{k=1}^t \vert \mu_k\vert  Q(y_k^N)=\sum_{k=1}^t \vert \mu_k\vert Q(y_k)>(1-\frac{\alpha}{2})\sum_{k=1}^t \vert \mu_k\vert=1-\frac{\alpha}{2}.$$
This implies that $z\in S(B_Y,Q,\frac{\alpha}{2})\subseteq S$. Now define $\phi:=T^N\colon Y\longrightarrow Y$ by the equation
$$\phi(a^N):=T(a)^N.$$
Similar estimates to the ones of the proof of Theorem~\ref{theo:polynomialWODPimpliesptensNWODP} for the case of $N$ odd, proves that
$$
\Vert \phi(z_i)-z_i\Vert<\frac{N^N}{N!}\bigl((1+\varepsilon)^{N+1}-1\bigr).
$$
Finally,
\[
\begin{split}
    \left\Vert \phi(z)-z'\right\Vert& =\left\Vert \sum_{k=1}^t \vert \mu_k\vert T(y_k)^N-\mu_k (x_k')^N\right\Vert\\
    & = \left\Vert \sum_{k=1}^t \mu_k \sign(\mu_k) T(y_k)^N-\mu_k (x_k')^N\right\Vert\\
    & = \left\Vert \sum_{k=1}^t \mu_k  (T(\sign(\mu_k) y_k)^N- (x_k')^N)\right\Vert\\
    & \leq \sum_{k=1}^t \vert \mu_k\vert \left\Vert T(\sign(\mu_k) y_k)^N- (x_k')^N\right\Vert
\end{split}
\]
Now, since $\sum_{k=1}^t \vert \mu_k\vert=1$, $\Vert T(\sign(\mu_k) y_k)-(x_k')^N\Vert<\varepsilon$ and from the estimates done in the proof of the case $N$ odd of Theorem~\ref{theo:polynomialWODPimpliesptensNWODP}, we get again that $$\Vert \phi(z)-z'\Vert<\frac{N^N}{N!}((1+\varepsilon)^{N+1}-1),$$ so we are done.
\end{proof}

\end{document}